\documentclass[11pt]{article}
	
	\usepackage{caption}
    \usepackage{url}
    \usepackage{verbatim}
	\usepackage{amsmath}
    \usepackage{color}
	\usepackage{amsthm}
	\usepackage{amsfonts}
	\usepackage{graphicx}
	\graphicspath{{figures/}{figures/eps/}{figures/pdf/}{figures/png/}}
	\usepackage{nicefrac}
    \usepackage{float}
	\usepackage{subfig}

	\def\captionfont{\setb@se{11pt}\protect\footnotesize}
    \def\captionfont{\protect\footnotesize}

    \newcommand{\iprd}[2]{\left( #1 , #2 \right)}

    \newcommand{\aiprd}[2]{a\!\left( #1 , #2 \right)}

    \newcommand{\Soh}{\mathring{S}_h}
    \newcommand{\bI}{{\bf I}}

	\def\norm#1#2{\left\| #1 \right\|_{#2}}

	\newcommand{\avephio}{\overline{\phi}_0}
	\newcommand{\cirphih}{\mathring{\phi}_h}
	\newcommand{\avemuh}{\overline{\mu}_h}
	\newcommand{\newtmuh}{\mathring{\mu}_{h,j}}
	\newcommand{\newtphih}{\mathring{\phi}_{h,j}}
	\newcommand{\cirmuh}{\mathring{\mu}_h}

	\newcommand{\dtau}{\delta_\tau}
	
	\newcommand{\phih}{\phi_h}

	\newcommand{\bv}{{\bf v}}

	\newcommand{\muh}{\mu_h}

	\newcommand{\bB}{{\bf B}}
	\newcommand{\bP}{{\bf P}}
	\newcommand{\bK}{{\bf K}}
	\newcommand{\bM}{{\bf M}}
	\newcommand{\bc}{{\bf c}}

	\def\tK{\tilde{\bf K}}
	\def\bI{{\bf I}}
	\def\bB{{\bf B}}
	\def\bC{{\bf C}}
	\def\bv{{\bf v}}
	\def\bV{{\bf V}}
	\def\d{\displaystyle}
	\newtheorem{thm}{Theorem}[section]

	\newtheorem{rem}[thm]{Remark}
	
\begin{document}
\title{A Robust Solver for a Second Order Mixed Finite Element Method for the
 Cahn-Hilliard Equation\thanks{The work of the first and third authors
 was supported in part by the National Science
  Foundation under Grant No. DMS-16-20273.}}
	\author{
Susanne C. Brenner\thanks{Department of Mathematics and Center for Computation \& Technology,
 Louisiana State University, Baton Rouge, LA 70803 (brenner@math.lsu.edu)},
	\and
Amanda E. Diegel\thanks{Department of Mathematics  and Center for Computation \& Technology,
 Louisiana State University, Baton Rouge, LA 70803 (diegel@math.lsu.edu, adiegel@math.msstate.edu)},
	\and
Li-Yeng Sung\thanks{Department of Mathematics and Center for Computation \& Technology,
 Louisiana State University, Baton Rouge, LA 70803 (sung@math.lsu.edu)}}

	\maketitle
	
	\numberwithin{equation}{section}
	
\begin{abstract}\noindent
 We develop a robust solver for a second order mixed finite element splitting scheme
 for the Cahn-Hilliard equation.  This work is an extension of our previous work in which we developed a robust solver for a first order mixed finite element splitting scheme for the Cahn-Hilliard equaion. The key ingredient of the solver is a preconditioned
 minimal residual algorithm (with a multigrid preconditioner) whose performance
 is independent of the spacial mesh size and the time step size for a given interfacial width parameter.  
 The dependence on the interfacial width parameter is also mild.
\end{abstract}

	\section{Introduction}\label{sec:Introduction}

The purpose of this paper is to demonstrate that the methods developed in our previous paper \cite{BDS:2018:CHMGSolver} can be extended to a second order (with respect to both time and space) mixed finite element method for the Cahn-Hilliard equation. Let $\Omega\subset \mathbb{R}^d$, $d=2,3$, be an open polygonal or polyhedral domain and consider the following form of the Cahn-Hilliard energy \cite{Cahn:1961}:
\begin{align}
 E(\phi) = \int_{\Omega} \left(\frac{1}{4 \varepsilon} (\phi^2-1)^2 +
 \frac{\varepsilon}{2} | \nabla \phi |^2\right) d{x},
\label{eq:ch-energy}
\end{align}
where  $\varepsilon >0$ is a constant, and $\phi \in H^1(\Omega)$ represents a concentration field. The phase equilibria are represented by $\phi = \pm 1$ and the parameter $\varepsilon$ represents a non-dimensional interfacial width between the two phases.
\par
The Cahn-Hilliard equation, which can be interpreted as the gradient flow of the energy \eqref{eq:ch-energy} in the dual space of $H^1(\Omega)$, is often represented in mixed form by
\begin{subequations}\label{subeqs:CH}
	\begin{eqnarray}
 \partial_t \phi = \varepsilon\Delta \mu  && \text{in} \,\Omega ,
	\label{eq:CH-mixed-a-alt}
	\\
 \mu = \varepsilon^{-1}\, \left(\phi^3 - \phi\right) -
 \varepsilon \Delta \phi && \text{in} \,\Omega,
	\label{eq:CH-mixed-b-alt}
	\end{eqnarray}
\end{subequations}
together with the boundary conditions $\partial_n \phi =0$ and $\partial_n \mu = 0$.
\par
Let $T$ be a positive number and  $H^{-1}_N(\Omega)$ be the dual space of $H^1(\Omega)$. A weak formulation of \eqref{eq:CH-mixed-a-alt}--\eqref{eq:CH-mixed-b-alt} is to find $(\phi,\mu)$ such that
\begin{subequations}\label{subeqs:Spaces}
	\begin{eqnarray}
\phi &\in& L^\infty\left(0,T;H^1(\Omega)\right)\cap
 L^4\left(0,T;L^\infty(\Omega)\right),	\\
\partial_t \phi &\in&  L^2\bigl(0,T; H_N^{-1}(\Omega)\bigr),
	\\
\mu &\in& L^2\bigl(0,T;H^1(\Omega)\bigr),
	\end{eqnarray}
\end{subequations}
and, for almost all $t\in (0,T)$,
\begin{subequations}\label{subeqs:Equations}
	\begin{align}
\langle \partial_t \phi ,\nu \rangle + \varepsilon \,\aiprd{\mu}{\nu}
 &= 0  \quad \forall \,\nu \in H^1(\Omega),
	\label{eq:weak-ch-a}\\
\iprd{\mu}{\psi}-\varepsilon \,\aiprd{\phi}{\psi} -
 \varepsilon^{-1}\iprd{\phi^3-\phi}{\psi} &= 0  \quad \forall \,\psi\in H^1(\Omega).
	\label{eq:weak-ch-b}
	\end{align}
\end{subequations}
Here $\langle \cdot, \cdot \rangle$ denotes the duality pairing between the spaces $H^{-1}_N(\Omega)$ and $H^1(\Omega)$, $(\cdot,\cdot)$ is the inner product of $L^2(\Omega)$, and
\begin{equation*}
  a({u},{v})=(\nabla u,\nabla v).
 \end{equation*}
The proof for the existence and uniqueness of the weak solution for \eqref{subeqs:Spaces}--\eqref{subeqs:Equations} with initial data
\begin{equation}\label{eq:InitialData}
 \phi(0)=\phi_0\in \,H^2_N(\Omega)=\{v\in H^2(\Omega):\,\partial v/\partial n=0
  \; \text{on}\; \partial\Omega\}
\end{equation}
can be found for example in \cite{Temam:1988:DynamicalSystem}.
\par
The Cahn-Hilliard equation is one of the most important and widely used equations in modeling two-phase phenomena. Originally developed to model phase separation of a binary alloy, often referred to as spinodal decomposition \cite{Cahn:1961,Cahn:1958,EZ:1986:CH}, variations of the Cahn-Hilliard equation have become popular components in modeling systems which describe physical processes such as two phase fluid flow, Hele-Shaw flows, copolymer fluids, crystal growth, and more (cf. \cite{CS:2018:CHPhaseField,CLSWW:2018:CHWillmore,CMW:2011:MinCahnHilliard,Feng:2006:NSCH,LLG:2002:HeleShaw,vTBVL:2009:PFCM} and the references therein). Due to the complexity of many of these systems along with their applications to physical models, the development of accurate, stable, and efficient numerical methods to solve the Cahn-Hilliard equation is still of high current interest (cf. \cite{AM:2018:FracCH,FGLWW:2018:FuncCH,FKLL:2018:FDCH,SongShu:2017:CHDG,YZWS:2017:3CompCH,WZZZ:2018:WeakGCH} and see \cite{BDS:2018:CHMGSolver} for earlier references).  Higher order numerical methods are important in this regard due to the accelerated convergence of these methods (cf. \cite{AKW:2018:SODGCahnHilliard, GGT:2014:TimeAdaptiveCH, GWWY:2016:H2FDCH, HBYT:2017:SOPF, LS:2015:HighOrdCH, YCWW:2018:BDFCahnHilliard, ZCHW:2015:TwoGridGH} and the references therein).
\par
In this paper, we consider a robust and efficient solver for the mixed finite element method for \eqref{eq:CH-mixed-a-alt}--\eqref{eq:CH-mixed-b-alt} developed in \cite{DWW:2016:SOCH}. The time discretization for this method is based on observing that the energy \eqref{eq:ch-energy} can be represented as the difference between two purely convex functionals \cite{Eyre:1998:ConvexSplitting}. In order to achieve unconditional stability along with second order in time convergence, a mixture of time stepping techniques is used when discretizing the equation relating to the chemical potential \eqref{eq:CH-mixed-b-alt}. In observing this equation, we note that the chemical potential is represented as the sum of three terms with regard to the phase field variable. We then treat each term as follows: a secant method defined by $\frac{\Psi(\phi^{m+1}) - \Psi(\phi^m)}{\phi^{m+1} - \phi^m}$ is applied to the cubic term where $\Psi(\phi) = \phi^4$, a second order Adams-Bashforth discretization is applied to the linear term, and a trapezoidal rule is applied to the advection term. The numerical method can then be described as a splitting scheme in time given by
\begin{align*}
\frac{\phi^{m+1} - \phi^m}{\tau} &= \varepsilon \Delta \mu^{m+\frac{1}{2}},
\\
\mu^{m+\frac{1}{2}} &= \frac{1}{4\varepsilon} \frac{\Psi(\phi^{m+1}) - \Psi(\phi^m)}{\phi^{m+1} - \phi^m} - \frac{1}{\varepsilon} \left(\frac{3}{2}\phi^m - \frac{1}{2}\phi^{m-1}\right) - \varepsilon \Delta\left(\frac{3}{4}\phi^{m+1} + \frac{1}{4}\phi^{m-1}\right),
\end{align*}
where $\tau$ is the time step size, and a spacial discretization that employs second order Lagrange finite elements. Fast solvers for other numerical schemes for the Cahn-Hilliard equation can be found in \cite{AKW:2013:DGMGCH,CLSWW:2018:CHWillmore,GuoXu:2014:CHDGsolvers,KayWelford:2006:MGCH,SKLK:2013:ParallelMG}.
\par
The remainder of this paper is organized as follows.  The mixed finite element method is introduced in Section~\ref{sec:SOMFEM}, followed by the construction and analysis of the  solver in Section \ref{sec:Solver}. Numerical results that demonstrate the performance of the solver are presented in Section \ref{sec:Numerics}, and we end the paper with some concluding remarks in Section~\ref{sec:Conclusion}.

\section{A Second Order Mixed Finite Element Method}
	\label{sec:SOMFEM}

Let $M$ be a positive integer, $0=t_0 < t_1 < \cdots < t_M = T$ be a uniform partition of $[0,T]$ and $\mathcal{T}_h$ be a quasi-uniform family of triangulations of $\Omega$ (cf. \cite{BScott:2008:FEM}). Furthermore, we consider the Lagrange finite element space $S_h\subset H^1(\Omega)$ given by
\begin{equation*}
  S_h=\{v\in C(\bar\Omega):\, v|_K \in {\mathcal P}_2(K)\,
  \forall \,\,  K\in \mathcal{T}_h \},
\end{equation*}
and define
\begin{equation*}
\Soh = S_h\cap L_0^2(\Omega),
\end{equation*}
where $L_0^2(\Omega)$ is the space of square integrable functions with zero mean.

The second-order (in time and space) splitting scheme for the Cahn-Hilliard equation we consider for the development of our robust solver is defined as follows \cite{DWW:2016:SOCH}:  for any $1\le m\le M-1$, given  $\phih^{m}, \phih^{m-1} \in S_h$, find $\phih^{m+1},\muh^{m+\frac12} \in S_h$ such that
	\begin{subequations}
	\label{subeqs:Schemea}
	\begin{align}
\iprd{\dtau \phih^{m+1}}{\nu} + \varepsilon \,\aiprd{\muh^{m+\frac12}}{\nu} &= \, 0  & \forall \, \nu \in S_h ,
	\label{eq:soch-scheme-a}
	\\
\varepsilon^{-1} \,\iprd{\chi\left(\phih^{m+1},\phih^m\right)}{\psi} - \varepsilon^{-1} \iprd{\frac{3}{2}\phih^m  - \frac{1}{2}\phih^{m-1} }{\psi}  & &
	\nonumber
	\\
+ \varepsilon \,\aiprd{\frac{3}{4}\phih^{m+1} + \frac{1}{4}\phih^{m-1}}{\psi}- \iprd{\muh^{m+\frac12}}{\psi} &= \, 0 & \forall \, \psi\in S_h,
	\label{eq:soch-scheme-b}
	\end{align}
	\end{subequations}
where
	\begin{align}
\dtau \phih^{m+1} &:= \frac{\phih^{m+1} - \phih^{m}}{\tau}, \quad \phih^{m+\frac12} := \frac12 \phih^{m+1} + \frac12 \phih^m, \quad \chi\left(\phih^{m+1},\phih^m\right) := \frac{1}{2}\left(\left(\phih^{m+1}\right)^2 + \left(\phih^{m}\right)^2\right)\phih^{m+\frac12}.
	\end{align}
Since this is a multi-step scheme, it requires a separate initialization process. For the first step, the scheme is as follows: given $\phih^0 \in S_h$, find $\phih^1, \muh^{\frac12} \in S_h$ such that
	\begin{subequations}\label{subeqs:SchemeInitial}
	\begin{align}
\iprd{\dtau \phih^{1}}{\nu} + \varepsilon \,\aiprd{\muh^{\frac12}}{\nu} &= \, 0  & \forall \, \nu \in S_h ,
	\label{eq:soch-scheme-a-initial}
	\\
\varepsilon^{-1} \,\iprd{\chi\left(\phih^{1},\phih^0\right)}{\psi} - \varepsilon^{-1} \iprd{\phih^0}{\psi} + \frac{\tau}{2} \, \aiprd{\muh^0}{\psi}  & &
	\nonumber
	\\
+ \varepsilon \,\aiprd{\phih^{\frac12}}{\psi} - \iprd{\muh^{\frac12}}{\psi} &= \, 0  & \forall \, \psi\in S_h,
	\label{eq:soch-scheme-b-initial}
	\end{align}
	\end{subequations}
where $\phih^0 := R_h \phi_0$, and $\muh^0 := R_h \mu_0$, such that $R_h:H^1(\Omega)\rightarrow S_h$ is the Ritz projection operator for the Neumann problem
defined by
\begin{subequations}\label{subeqs:Rh}
\begin{align}
  a(R_hv-v,w)&=0 \qquad \forall\,w\in S_h, \\
  (R_h v-v,1)&=0.  \label{eq:RhMean}
\end{align}
\end{subequations}
and
	\begin{equation}
\mu_0 := \varepsilon^{-1}\left(\phi_0^3 - \phi_0\right) - \varepsilon\Delta\phi_0.
	\end{equation}

\begin{rem}
\label{rem:energy-laws}
It is important to note that the initialization scheme follows a similar energy law as that of \eqref{eq:ch-energy} and the second order finite element method \eqref{eq:soch-scheme-a}--\eqref{eq:soch-scheme-b} satisfies a modification of this energy law. Let $(\phih^{1}, \muh^{\frac12}) \in S_h\times S_h$ be the unique solution of the initialization scheme \eqref{eq:soch-scheme-a-initial} -- \eqref{eq:soch-scheme-b-initial} and let $(\phih^{m+1}, \muh^{m+\frac12}) \in S_h\times S_h$ be the unique solution of  \eqref{eq:soch-scheme-a} -- \eqref{eq:soch-scheme-b}.  Then the following energy laws hold for any $h,\,  \tau >0$ \cite{DWW:2016:SOCH}:
	\begin{align}
E\left(\phih^{1}\right) + \tau \varepsilon \norm{\nabla\muh^{\frac12}}{L^2}^2 &+ \frac{1}{4\varepsilon} \norm{\phih^1 - \phih^0}{L^2}^2 \le E\left(\phih^0\right) + \frac{ \varepsilon \tau^2}{4} \norm{\Delta_h \muh^0}{L^2}^2,
	\label{eq:soch-ConvSplitEnLaw-initial}
	\\
F\left(\phih^{\ell+1}, \phih^{\ell}\right) +\tau \varepsilon \sum_{m=1}^\ell \norm{\nabla\muh^{m+\frac12}}{L^2}^2 &+  \sum_{m=1}^\ell \Bigg[ \frac{1}{4\varepsilon} \norm{\phih^{m+1} - 2 \phih^m + \phih^{m-1}}{L^2}^2
	\nonumber
	\\
&+ \frac{\varepsilon}{8} \norm{\nabla \phih^{m+1} - 2 \nabla \phih^m + \nabla \phih^{m-1}}{L^2}^2 \Biggr] = F\left(\phih^1, \phih^0\right),
	\label{eq:soch-ConvSplitEnLaw}
	\end{align}
for all $1 \leq \ell \leq M-1$ where $E(\phi)$ is defined in \eqref{eq:ch-energy} and $F(\phi, \psi)$ is defined as
	\begin{equation}
F(\phi, \psi) := E(\phi) + \frac{1}{4\varepsilon} \norm{\phi - \psi}{L^2}^2 + \frac{\varepsilon}{8} \norm{\nabla \phi - \nabla \psi}{L^2}^2.
\label{eq:modified-so-energy}
	\end{equation}
\end{rem}

\begin{rem}\label{rem:stability}
The energy laws in Remark \ref{rem:energy-laws} are key properties of the solution of \eqref{subeqs:Schemea}--\eqref{subeqs:SchemeInitial}. It can be shown {\rm\cite{DWW:2016:SOCH}} that these energy laws lead to the unconditional stability estimates \newline $\phih \in L^{\infty}(0,T;L^{\infty}(\Omega))$ and $\muh \in L^{\infty}(0,T;L^{2}(\Omega))$.
Moreover, under the assumption that \newline $\phi \in  L^{\infty}(0,T;W^{1,6}(\Omega)) \cap H^1(0,T;H^3(\Omega)) \cap H^2(0,T;H^3(\Omega)) \cap H^3(0,T;L^2(\Omega)), \phi^2 \in H^2(0,T;H^1(\Omega))$, $\mu \in L^{2}(0,T;H^3(\Omega))$, and $ 0 \le \tau \le \tau_0$ for a sufficiently small $\tau_0$, the error estimate
\begin{align}\label{eq:ErrorEstimate}
\max\limits_{1 \le m \le M} \norm{\nabla \phi(m\tau) - \nabla \phih^m}{L^2}^2
+ \tau \sum_{m=1}^{M-1} \norm{\nabla \mu\left((m+\nicefrac{1}{2})\tau\right) - \nabla \muh^{m+\frac12}}{L^2}^2
\le C(\varepsilon,T) (\tau^4 + h^4)
\end{align}
holds for a positive constant $C$ that depends on $\varepsilon$ and $T$ but does not depend on $\tau$ and $h$.
 \end{rem}

A key attribute to the development of the solver in \cite{BDS:2018:CHMGSolver} was the establishment of an equivalent numerical method utilizing mean zero functions and we now show the extension to the second order finite element method. Specifically, it follows from \eqref{eq:soch-scheme-a} that $(\phih^{m+1},1) = (\phi_0,1)$ for $0\le m \le M-1$, and hence,
\begin{align}
\phih^{m+1} = \avephio  + \cirphih^{m+1} \quad \text{for} \quad 0\leq m \leq M-1,
\label{eq:ave-phi-def}
\end{align}
where $\avephio = {\iprd{\phi_0}{1}}/{\iprd{1}{1}}$ is the mean of
 $\phi_0$ over $\Omega$ and $\cirphih \in \Soh$.
 We can also write
\begin{equation}\label{eq:ave-mu-def}
\muh^{m+\frac12} = \avemuh^{m+\frac12} + \cirmuh^{m+\frac12},
\end{equation}
 where $\avemuh^{m+\frac12}$ is a constant function and $\cirmuh^{m+\frac12} \in \Soh$.
\par
 Using \eqref{eq:ave-phi-def} and \eqref{eq:ave-mu-def}, we can rewrite
 \eqref{eq:soch-scheme-a}--\eqref{eq:soch-scheme-b}  in the following equivalent form:
  For $1\le m\le M-1$, find $\cirphih^{m+1},\cirmuh^{m+1}\in \Soh$ such that
\begin{subequations}\label{subeqs:Schemeb}
\begin{align}
\iprd{\dtau \cirphih^{m+1}}{\nu} + \varepsilon \,\aiprd{\cirmuh^{m+\frac12}}{\nu} &= \, 0  & \forall \, \nu \in \Soh ,
\label{eq:chmz-scheme-a}
	\\
\varepsilon^{-1} \,\iprd{\chi\left(\cirphih^{m+1}+\avephio,\cirphih^m+\avephio\right)}{\psi} - \varepsilon^{-1} \iprd{\frac{3}{2}\cirphih^m  - \frac{1}{2}\cirphih^{m-1} }{\psi}  & &
	\nonumber
	\\
+ \varepsilon \,\aiprd{\frac{3}{4}\cirphih^{m+1} + \frac{1}{4}\cirphih^{m-1} }{\psi}- \iprd{\cirmuh^{m+\frac12}}{\psi} &= \, 0 & \forall \, \psi\in \Soh,
	\label{eq:chmz-scheme-b}
\end{align}
\end{subequations}
 where
\begin{equation*}
\dtau \cirphih^{m+1} = \frac{\cirphih^{m+1}-\cirphih^{m}}{\tau} .
\end{equation*}
Note that  we can recover the constant function $\avemuh^{m+\frac12}$  from $\cirphih^{m+1}$ and $\cirphih^{m}$
  through the relation
$$\iprd{\avemuh^{m+\frac12}}{1} = \varepsilon^{-1} \big(\chi\left(\cirphih^{m+1}+\avephio,\cirphih^m+\avephio\right) - \avephio ,{1}\big)$$
which follows from \eqref{eq:soch-scheme-b}, \eqref{eq:ave-phi-def} and \eqref{eq:ave-mu-def}.
 \begin{rem}\label{rem:unique-solvability}
 The nonlinear system \eqref{subeqs:Schemeb}
 is uniquely solvable for any mesh parameters $h, \tau$ and for any model parameters. The proof is based on convexity arguments and follows in a similar manner as that of Theorem 5 from \cite{HWWL:2009:PFCE}.
\end{rem}
\section{A Robust Solver}\label{sec:Solver}
 We will solve the nonlinear system \eqref{subeqs:Schemeb} by Newton's iteration.  Let
 $(\newtphih^{m+1},\newtmuh^{m+1})\in \Soh\times\Soh$ be the output of the $j$-th step.
    In order to advance the iteration, we need to find
 $(\delta_j \mathring{\mu}, \delta_j \mathring{\phi}) \in \Soh \times \Soh$ such that
\begin{subequations}\label{subeqs:Jacobian1}
\begin{alignat}{3}
\tau  \varepsilon \, a({\delta_j \mathring{\mu}},{\nu}) + (\nu,\delta_j \mathring{\phi})
 &= F_j(\nu) &\quad& \forall\, \nu \in \Soh,\\
 ({\delta_j \mathring{\mu}},{\psi}) - \left[ \frac{1}{4\varepsilon}
 \Bigg({\left(3(\phi_{h,j}^{m+1})^2 + 2 \phi_{h,j}^{m+1}\phi_{h}^{m} + (\phi_{h}^{m})^2\right)\delta_j \mathring{\phi}},{\psi}\Bigg)
 + \frac{3\varepsilon}{4}\,a({\delta_j \mathring{\phi}},{\psi})\right]
 &= G_j(\psi) &\quad& \forall\, \phi \in \Soh,
\end{alignat}
\end{subequations}
where $\phi_{h,j}^{m+1} = \mathring{\phi}_{h,j}^{m+1} + \avephio$ and
\begin{subequations}\label{subeqs:Jacobian2}
\begin{align}
 F_j(\nu) &= \tau \varepsilon \, a({\mathring{\mu}_{h,j}^{m+\frac12}},{\nu} )
+ ({\mathring{\phi}_{h,j}^{m+1} - \mathring{\phi}_{h}^{m}},{\nu}),
\\
G_j(\psi) &= ({\mathring{\mu}_{h,j}^{m+\frac12}},{\psi})-
\Bigg[ \varepsilon^{-1} \,\iprd{\chi\left(\phi_{h,j}^{m+1},\phi_{h}^m\right)}{\psi} - \varepsilon^{-1} \iprd{\frac{3}{2}\mathring{\phi}_{h}^m  - \frac{1}{2}\mathring{\phi}_{h}^{m-1} }{\psi}  & &
	\nonumber
	\\
&\quad+ \varepsilon \,\aiprd{\frac{3}{4}\mathring{\phi}_{h,j}^{m+1} + \frac{1}{4}\mathring{\phi}_{h}^{m-1} }{\psi} \Bigg].
\end{align}
\end{subequations}
 The next output of the Newton iteration is then given by
\begin{align}\label{eq:NextOutput}
  (\mathring{\mu}_{h,j+1}^{m+\frac12},\mathring{\phi}_{h,j+1}^{m+1})
= (\mathring{\mu}_{h,j}^{m+\frac12},\mathring{\phi}_{h,j}^{m+1}) -
       (\delta_j \mathring{\mu}, \delta_j \mathring{\phi}).
\end{align}

 The first challenge we must overcome is the inconvenient zero mean constraint.  We circumvent this constraint by reformulating
 \eqref{subeqs:Jacobian1} as the following equivalent problem:
 Find $(\delta_j \mu, \delta_j \phi ) \in S_h \times S_h$ such that
\begin{subequations}\label{subeqs:NewJacobian1}
\begin{alignat}{4}
\tau  \varepsilon \, \Big[ a({\delta_j \mu},{\nu}) +
\iprd{\delta_j \mu}{1} \iprd{\nu}{1} \Big] + (\nu,\delta_j \phi)
 = \tilde{F}_j(\nu) &\quad & \forall\, \nu \in S_h,\\
 ({\delta_j \mu},{\psi}) - \frac{3\varepsilon}{4} \iprd{\delta_j \phi}{1} \iprd{\psi}{1}&\qquad &
 \nonumber
 \\
 - \left[ \frac{1}{4\varepsilon} \Bigg({\left(3(\phi_{h,j}^{m+1})^2 + 2 \phi_{h,j}^{m+1}\phi_{h}^{m} + (\phi_{h}^{m})^2\right)\delta_j \phi},{\psi}\Bigg)
 + \frac{3\varepsilon}{4}\,a({\delta_j \phi},{\psi})\right]
 = \tilde{G}_j(\psi) &\quad & \forall\, \psi \in S_h,
\end{alignat}
\end{subequations}
 where
\begin{equation}\label{eq:NewJacobian2}
\tilde{F}_j(\nu) =
\begin{cases}
F_j(\nu) & \text{if } \nu \in \Soh \\
0 & \text{if } \nu=1
\end{cases}
\quad\text{and}\quad
\tilde{G}_j(\psi) =
\begin{cases}
G_j(\psi) & \text{if } \psi \in \Soh \\
0 & \text{if } \psi=1
\end{cases}.
\end{equation}
\par
\begin{rem}\label{rem:Equivalence}
 It is easy to check that both \eqref{subeqs:Jacobian1} and
 \eqref{subeqs:NewJacobian1} are well-posed
 linear systems and that the solution
 $(\delta_j \mathring{\mu}, \delta_j \mathring{\phi})$ of
 \eqref{subeqs:Jacobian1}--\eqref{subeqs:Jacobian2} also satisfies
 \eqref{subeqs:NewJacobian1}--\eqref{eq:NewJacobian2}.
\end{rem}
\par
 Under the change of variables
\begin{equation*}
\iprd{\delta_j \mu}{\nu}  \rightarrow (4\tau)^{-\frac{1}{4}}
\varepsilon^{-\frac{1}{2}} \iprd{\delta_j \mu}{\nu}
 \quad\text{and} \quad
 \iprd{\delta_j \phi}{\psi}  \rightarrow (4\tau)^{\frac{1}{4}}
 \varepsilon^{\frac{1}{2}} \iprd{\delta_j \phi}{\psi},
\end{equation*}
 the system \eqref{subeqs:NewJacobian1} becomes
\begin{subequations}\label{subeqs:CV}
\begin{align}
&\frac{\tau^{\frac{1}{2}}}{2} \Big[ a({\delta_j \mu},{\nu}) +
\iprd{\delta_j \mu}{1} \iprd{\nu}{1} \Big] + (\nu,\delta_j \phi)
 = \tilde{F}_j((4\tau)^{-\frac{1}{4}}
 \varepsilon^{-\frac{1}{2}} \nu),\\
 &({\delta_j \mu},{\psi}) - \frac{3\tau^{\frac12} \varepsilon^2}{2} \iprd{\delta_j \phi}{1} \iprd{\psi}{1}&&&
 \nonumber
 \\
 &\quad- \left[ \frac{ \tau^{\frac12}}{2} \Bigg({\left(3(\phi_{h,j}^{m+1})^2 + 2 \phi_{h,j}^{m+1}\phi_{h}^{m} + (\phi_{h}^{m})^2\right)\delta_j \phi},{\psi}\Bigg)
 + \frac{3 \tau^{\frac12}\varepsilon^2}{2}\,a({\delta_j \phi},{\psi})\right]
= \tilde{G}_j((4\tau)^{\frac{1}{4}} \varepsilon^{\frac{1}{2}}\psi),
\end{align}
	\label{eq:unconstrained-final}
\end{subequations}
 for all $(\nu,\psi)\in S_h\times S_h$.
\par
 Let $n_h$ be the dimension of $S_h$ and
 $\varphi_1,\ldots,\varphi_{n_h}$  be the standard nodal basis functions for $S_h$.
 The system matrix for \eqref{subeqs:CV} is given by
\begin{equation}\label{eq:SystemMatrix}
\begin{bmatrix}
\frac{\tau^\frac{1}{2}}{2} \left(\bK + \bc \bc^t\right) & \bM \\
\bM & -\frac{\tau^\frac{1}{2}}{2} {\bf J} (\phi_{h,j}^{m+1}) -
\frac{3\tau^\frac{1}{2}\varepsilon^2}{2} \left(\bK + \bc \bc^t\right)
\end{bmatrix},
\end{equation}
 where the stiffness matrix $\bf K$ is  defined by
 ${\bf K}(k,\ell)=(\nabla\varphi_k,\nabla\varphi_\ell)$,
 the mass matrix $\bf M$ is  defined by
  ${\bf M}(k,\ell)=(\varphi_k,\varphi_\ell)$, the vector $\bf c$ is  defined by
 ${\bf c}(k)=(\varphi_k,1)$, and the matrix ${\bf J} (\phi_{h,j}^{m+1})$ is defined by
   $${\bf J} (\phi_{h,j}^{m+1})(k,\ell)= \Bigg(\left(3(\phi_{h,j}^{m+1})^2 + 2 \phi_{h,j}^{m+1}\phi_{h}^{m} + (\phi_{h}^{m})^2\right)\varphi_k,\varphi_\ell\Bigg).$$
\par
Note that, since the mixed finite element method is convergent, we can expect
 $(\phi^{m+1}_{h,j})^2, \phi_{h,j}^{m+1}\phi_{h}^{m}$ and $(\phi^m_h)^2$ to be close to 1 away from an interfacial region with width $\varepsilon$.
 Therefore, for small $\varepsilon$, we can
  replace ${\bf J} (\phi_{h,j}^m)$ by $6{\bf M}$ in \eqref{eq:SystemMatrix}.
  The following result is motivated
  by this observation.
\begin{thm}\label{thm:Preconditioner}
 Let the matrices $\bf B$ and $\bf P$ be defined by
\begin{align}
{\bf B}&=\begin{bmatrix}
\frac{\tau^\frac{1}{2}}{2} (\bK + \bc \bc^t) & \bM \\
\bM & -\frac{6\tau^\frac{1}{2}}{2} \bM - \frac{3\tau^\frac{1}{2}\varepsilon^2}{2} (\bK + \bc \bc^t)
\end{bmatrix},\label{eq:block-system-simple}\\
\bP &=
 \begin{bmatrix}
\frac{\tau^\frac{1}{2}}{2}( \bK+\bc\bc^t) + \bM & {\bf 0} \\
{\bf 0} & \frac{\tau^\frac{1}{2}}{2} \varepsilon^2 (\bK+\bc\bc^t)+ \bM
\end{bmatrix},
\label{eq:preconditioner}
\end{align}
 where $0\leq \tau,\varepsilon\leq 1$.
 There exist two positive constants $C_1$ and $C_2$ independent of
 $\varepsilon$, $h$ and $\tau$ such that
\begin{equation}\label{eq:eigenvalue-condition}
C_2 \max(\tau^\frac{1}{2},\varepsilon) \le|\lambda | \le C_1 \quad
\text{\rm for any eigenvalue } \lambda \text{ of } \bP^{-1} {\bf B}.
\end{equation}
\end{thm}
\begin{proof} A simple calculation shows that
\begin{align*}
 \bP^{-1}\bB&=\left(\begin{bmatrix} \bM & 0\\ 0&\bM\end{bmatrix}
     \begin{bmatrix} \frac{\tau^\frac{1}{2}}{2} \tK +\bI& 0 \\
      0 &\frac{\tau^\frac{1}{2}}{2}\varepsilon^2\tK+\bI
 \end{bmatrix}\right)^{-1}
 \left(\begin{bmatrix} \bM & 0\\ 0&\bM\end{bmatrix}
\begin{bmatrix}
   \frac{\tau^\frac{1}{2}}{2}\tK &\bI\\
   \bI&-\frac{6\tau^\frac{1}{2}}{2}\bI-\frac{3\tau^\frac{1}{2}\varepsilon^2}{2}\tK
 \end{bmatrix}\right)\\
   &=\begin{bmatrix} \frac{\tau^\frac{1}{2}}{2} \tK +\bI& 0 \\
      0 &\frac{\tau^\frac{1}{2}}{2}\varepsilon^2\tK+\bI
 \end{bmatrix}^{-1}
 \begin{bmatrix}
   \frac{\tau^\frac{1}{2}}{2}\tK &\bI\\
   \bI&-\frac{6\tau^\frac{1}{2}}{2}\bI-\frac{3\tau^\frac{1}{2}\varepsilon^2}{2}\tK
 \end{bmatrix},
\end{align*}
 where $\tK=\bM^{-1}(\bK+\bc\bc^t)$ and $\bI$ is the $n_h\times n_h$ identity matrix.
\par
 By the spectral theorem, there exist $\bv_1,\ldots,\bv_{n_h}\in\mathbb{R}^{n_h}$ and
 positive numbers $\kappa_1,\ldots,\kappa_{n_h}$ such that
\begin{equation*}
  \tK \bv_j=\kappa_j \bv_j \qquad\text{for}\quad 1\leq j\leq n_h
\end{equation*}
 and
\begin{equation*}
     \bv_j^t\bM\bv_\ell=\begin{cases} 1 &\quad \text{if $j=\ell$}\\
           0 &\quad\text{if $j\neq\ell$} \end{cases} \;.
\end{equation*}
\par
 Observe that the two dimensional space  $\bV_j$ spanned by
\begin{equation*}
  \begin{bmatrix} \bv_j \\ 0 \end{bmatrix} \quad \text{and} \quad
  \begin{bmatrix} 0 \\ \bv_j\end{bmatrix}
\end{equation*}
 is invariant under $\bP^{-1}\bB$ and
\begin{equation*}
 \bP^{-1}\bB\left(\alpha\begin{bmatrix} \bv_j \\ 0 \end{bmatrix} +\beta
     \begin{bmatrix} 0 \\ \bv_j\end{bmatrix}\right)=\gamma
     \begin{bmatrix} \bv_j \\ 0 \end{bmatrix} +\delta
     \begin{bmatrix} 0 \\ \bv_j\end{bmatrix},
\end{equation*}
 where
\begin{equation*}
  \begin{bmatrix} \gamma\\ \delta \end{bmatrix}=
  \begin{bmatrix} \frac{\tau^\frac{1}{2}}{2}\kappa_j+1&0\\ 0&\frac{\tau^\frac{1}{2}}{2}\varepsilon^2\kappa_j+1
  \end{bmatrix}^{-1}
  \begin{bmatrix} \frac{\tau^\frac{1}{2}}{2}\kappa_j & 1\\ 1& -\frac{6\tau^\frac{1}{2}}{2}
  -\frac{3\tau^\frac{1}{2}\varepsilon^2}{2}\kappa_j \end{bmatrix}
  \begin{bmatrix}\alpha \\ \beta \end{bmatrix}.
\end{equation*}
\par
 It follows that the eigenvalues of $\bP^{-1}\bB$ are precisely the eigenvalues of the matrix
\begin{align*}
\bC_j&=\begin{bmatrix}
   \frac{\tau^\frac{1}{2}}{2}\kappa_j+1&0\\ 0&\frac{\tau^\frac{1}{2}}{2}\varepsilon^2\kappa_j+1
   \end{bmatrix}^{-1}
     \begin{bmatrix}
     \frac{\tau^\frac{1}{2}}{2}\kappa_j & 1\\ 1& -\frac{6\tau^\frac{1}{2}}{2}
  -\frac{3\tau^\frac{1}{2}\varepsilon^2}{2}\kappa_j
     \end{bmatrix}\\
   &=\begin{bmatrix}
    \d  \frac{\frac{\tau^\frac{1}{2}}{2}\kappa_j }{\frac{\tau^\frac{1}{2}}{2}\kappa_j+1}&\d
    \frac{1}{\frac{\tau^\frac{1}{2}}{2}\kappa_j+1}\\[20pt]
   \d  \frac{1}{\frac{\tau^\frac{1}{2}}{2}\varepsilon^2\kappa_j+1}& \d
 \frac{-6\frac{\tau^\frac{1}{2}}{2}-3\frac{\tau^\frac{1}{2}}{2}\varepsilon^2\kappa_j }{\frac{\tau^\frac{1}{2}}{2}\varepsilon^2\kappa_j+1}
      \end{bmatrix}
\end{align*}
 for $1\leq j\leq n_h$.  Hence we only need to understand the behavior
 of the eigenvalues of the matrix
\begin{equation*}
 \bC=\begin{bmatrix}
  \d  \frac{\omega}{\omega+1}&\d  \frac{1}{\omega+1}\\[20pt]
          \d  \frac{1}{\omega\varepsilon^2+1}& \d
          \frac{-3\tau^\frac12-3\omega\varepsilon^2 }{\omega\varepsilon^2+1}
      \end{bmatrix},
\end{equation*}
 where $\omega$ is a positive number and $0<\tau,\varepsilon\leq 1$.
\par
 First of all we have
\begin{equation}\label{eq:P1}
  |\lambda|\leq \|\bC\|_\infty \leq 4
\end{equation}
 for any eigenvalue $\lambda$ of $\bC$, which implies that the second estimate in
 \eqref{eq:eigenvalue-condition} holds for $C_1=4$.
\par
  A direct calculation shows that
\begin{equation*}
 |\det\bC|=\frac{1+3\tau^\frac12\omega+3\varepsilon^2\omega^2}
 {1+(1+\varepsilon^2)\omega+\varepsilon^2\omega^2}
 \geq \frac{1+3\tau^\frac12\omega+3\varepsilon^2\omega^2}
 {1+2\omega+\varepsilon^2\omega^2}.
\end{equation*}
 On one hand we have
\begin{equation*}
 1+2\omega+\varepsilon^2\omega^2
   \leq \tau^{-\frac12}(1+3\tau^\frac12\omega+3\varepsilon^2\omega^2),
\end{equation*}
 which implies
\begin{equation}\label{eq:P2}
  |\det\bC|\geq \tau^\frac12.
\end{equation}
 On the other hand we also have
\begin{equation*}
  1+2\omega+\varepsilon^2\omega^2\leq
  \varepsilon^{-1}(1+2\varepsilon\omega+\varepsilon^2\omega^2)
  \leq 2\varepsilon^{-1}(1+\varepsilon^2\omega^2)
   \leq 2\varepsilon^{-1}(1+3\tau^\frac12\omega+3\varepsilon^2\omega^2),
\end{equation*}
 which implies
\begin{equation}\label{eq:P3}
|\det\bC|\geq \frac{\varepsilon}{2}.
\end{equation}
\par
 Putting \eqref{eq:P1}--\eqref{eq:P3} together we see that
\begin{equation*}
  4|\lambda|\geq|\det \bC|\geq \max(\tau^\frac12,\varepsilon/2)
\end{equation*}
 for any eigenvalue $\lambda$ of $\bC$.  Therefore, the first estimate in
 \eqref{eq:eigenvalue-condition} holds with $C_2=1/8$.
\end{proof}
\begin{rem}\label{rem:sharp}
We note that the preconditioner $\bP$ can also be analyzed by the theory in \cite{Zulehner:2011:RobustSaddlePt}. At the same time, the estimate \eqref{eq:eigenvalue-condition} obtained by our elementary approach is already sharp (cf. the discussion in \cite{BDS:2018:CHMGSolver}).
\end{rem}
\par
 In our numerical experiments, we use the preconditioner $\bP_*$ given by
\begin{equation}\label{eq:PStar}
 \bP_*= \begin{bmatrix}
\frac{\tau^\frac{1}{2}}{2}\bK + \bM & {\bf 0} \\
{\bf 0} & \frac{\tau^\frac{1}{2}}{2} \varepsilon^2 \bK+ \bM
\end{bmatrix}.
\end{equation}
 Since the two symmetric positive definite
 matrices $\bP$ and $\bP_*$ are spectrally equivalent, we immediately deduce from
 Theorem~\ref{thm:Preconditioner} that there exist two positive constants $C_3$ and $C_4$ independent of
 $\varepsilon$, $h$ and $\tau$ such that
\begin{equation}\label{eq:PrecondEst}
   C_4\max(\tau^\frac12,\varepsilon)\leq |\lambda|\leq C_3
\end{equation}
 for any eigenvalue $\lambda$ of $\bP_*^{-1}\bB$.
\par
 According to \eqref{eq:PrecondEst}, the performance of the  preconditioned
 MINRES algorithm (cf. \cite{Greenbaum:1997:Iterative,ESW:2014:FEFIS})
 for systems involving $\bf B$ is independent of $\tau$ and $h$ for a given $\varepsilon$,
 and also independent of $\varepsilon$ and $h$ for a given $\tau$.
  Similar behavior can also be expected for systems involving the
 matrix in \eqref{eq:SystemMatrix}.  Furthermore, the action of $(\gamma \bK + \bM)^{-1}$
  on a vector can be computed by a multigrid method, which creates large computational savings.
\begin{rem}\label{rem:varepsilon}
 Recall that matrix $\bB$ is obtained from the matrix in \eqref{eq:SystemMatrix} by replacing
 ${\bf J}(\phi_{h,j}^m)$ by $6\bM$ and its justification depends on $\varepsilon$.
 Therefore we expect to see some dependence of
 the performance of the preconditioned MINRES algorithm on $\varepsilon$
 for a given $\tau$.
\end{rem}
\begin{rem}\label{rem:SmallTau}
   When $\tau$ becomes $0$, the matrix
     $$\bB=\begin{bmatrix}
       0 & \bM \\
       \bM & 0
     \end{bmatrix}$$
 is well-conditioned. Therefore the performance of the preconditioned MINRES algorithm
 for systems involving the matrix in \eqref{eq:SystemMatrix}
 will improve as the time step size decreases.
\end{rem}
\begin{rem}\label{rem:SPP}
 Block diagonal preconditioners for saddle point systems are discussed in
 {\rm\cite{BGL:2005:SaddlePoint,MW:2011:SaddlePoint}} and the references therein.
\end{rem}

\section{Numerical Experiments}
\label{sec:Numerics}
 In this section, we report the results of
 several numerical experiments in two and three dimensions. All computations were carried out using the FELICITY MATLAB/C++ Toolbox \cite{Walker:2018:felicity}.
\par
 In the first four numerical experiments, we solve \eqref{subeqs:Schemea} on the
 unit square $\Omega=(0,1)^2$ using uniform meshes. The initial mesh $\mathcal{T}_0$ is generated by the two diagonals of $\Omega$  and the meshes $\mathcal{T}_1,\mathcal{T}_2,\ldots$ are obtained from $\mathcal{T}_0$ by uniform refinements. The system \eqref{subeqs:Schemea} (or equivalently \eqref{subeqs:Schemeb})
 is solved by the Newton iteration with a tolerance of $10^{-15}$ for $\|\delta_j\phi\|_{L^\infty(\Omega)}$ or
  a residual tolerance of $10^{-7}$ for \eqref{subeqs:NewJacobian1}--\eqref{eq:NewJacobian2},
  whichever is satisfied first.  It turns out that only one Newton iteration is needed for each time step in all the experiments.
\par
  During each Newton iteration, the systems involving
  \eqref{eq:SystemMatrix} are solved by a preconditioned MINRES
  algorithm with a residual tolerance of $10^{-7}$.  The systems
  involving the preconditioner $\bf P$ are solved by
  a multigrid $V(2,2)$ algorithm that uses the Gauss-Seidel iteration as the smoother
  (cf. \cite{Hackbusch:1985:MMA,TOS:2001:MG}).
  In all our experiments, the maximum number of preconditioned MINRES iterations occured
  during the first few time steps after which the number of iterations
  would decrease and level off.
\par
  In the first experiment, we use the initial data
\begin{equation}
\phi_{h}^0 = \mathcal{I}_h
\Big[\Big(\frac{1}{2}\Big)[1 - \cos(4\pi x_1)][1 - \cos(2\pi x_2))]- 1\Big] ,
\end{equation}
 where $\mathcal{I}_h :   H^2(\Omega) \longrightarrow S_h$ is the
 standard nodal interpolation operator. We take  $\tau = {0.002}/{64}$
 with a final  time $T=0.04$ and an interfacial width parameter of
 $\varepsilon = 0.05$. In Table~\ref{ch-tab-fixed-tau-eps}, we report the average number of preconditioned MINRES
 iterations over all time steps along with the average solution time per time step as the mesh is refined. In addition, we display the factor of increase in the average time to solve per time step from the previous mesh size to the current mesh size. (The timing mechanism is the `tic toc' command in MATLAB.)
Observe that the performance of the preconditioned MINRES algorithm does not depend on $h$ and the
 solution time per time step grows linearly with the number of degrees of freedom.
 \par
 We then run the same test using MATLAB's built in backslash command to solve. Due to MATLAB's built-in efficiency standards, MATLAB's backslash command outperforms the solver described in this paper on coarse mesh sizes. However, as the mesh is refined, one does see that the time to solve using MATLAB's backslash command approaches the quadratic growth one expects from using a solver such as LU decomposition. By comparison, the performance of the method outlined in this paper continues to grow linearly as the mesh size is refined and the advantage is clearly observed by comparing the performance of the two solvers for the mesh size $h = \nicefrac{1}{512}$.
\begin{table}[H]
	\centering
	\begin{tabular}{c | ccc | cc}
 & \multicolumn{3}{ c }{MINRES Solver} &\multicolumn{2}{| c }{MATLAB's Backslash}
 \\
$h$ & MINRES Its. & Time to Solve (s) & Factor of Inc. & Time to Solve (s) & Factor of Inc.
	\\
	\hline
$\nicefrac{1}{8}$ & 23 & $2.07 \times 10^{-2}$ &  ~ & $5.46 \times 10^{-3}$ & ~
	\\
$\nicefrac{1}{16}$ & 26 & $5.74 \times 10^{-2}$ &2.773 & $1.85 \times 10^{-2}$ & 3.386
	\\
$\nicefrac{1}{32}$ & 38 & $3.01 \times 10^{-1} $ & 5.233 & $9.75 \times 10^{-2}$ & 5.268
	\\
$\nicefrac{1}{64}$ & 48 & $1.44 \times 10^0 $ & 4.785 & $ 5.09 \times 10^{-1}$ & 5.219
	\\
$\nicefrac{1}{128}$ & 52 & $ 6.47 \times 10^0$ & 4.499 & $2.76 \times 10^0$  & 5.432
\\
$\nicefrac{1}{256}$ & 55 & $3.44 \times 10^1$ & 5.318 &  $1.69 \times 10^1$ & 6.143
\\
$*\nicefrac{1}{512}$ & 57 & $1.85 \times 10^2$ & 5.366  & $6.42 \times 10^2$ & 37.83
\\
\hline
\end{tabular}
\caption{The average number of preconditioned MINRES iterations over all
 time steps together with the average solution time per time step as the mesh
 is refined ($\Omega = (0,1)^2, \tau = \nicefrac{0.002}{64}$, $T =0.04$
 $\varepsilon=0.05$.) The star above indicates that the final stopping time $T$ was cut short for the test run utilizing MATLAB's backslash command due to the large computational time. }
	\label{ch-tab-fixed-tau-eps}
\end{table}
\par
The purpose of the second experiment is to compare the performance of the solver developed for the second order finite element method presented in this paper with that of the solver developed for the first order finite element method presented in \cite{BDS:2018:CHMGSolver}. We choose an initial condition of the oval described by $$\phi_{h}^0 = \mathcal{I}_h
\left[-1.01\tanh\left(\frac{\nicefrac{(x_1-0.5)^2)}{0.075}+ \nicefrac{(x_2-0.5)^2}{0.05}-1}{2\sqrt{\varepsilon}}\right)\right]$$ as shown in the Figure \ref{fig:oval_test}. We fix $\varepsilon=0.03$, $\tau = \nicefrac{0.07}{256}$, and a final stopping time of $T=0.7$.

It is well known that exact solutions are difficult to construct for the two dimensional Cahn-Hilliard equation without adding an artificial source term. Therefore, in order to obtain an accurate comparison of the two solvers, we have chosen the solution to the second order scheme with a mesh size of $h=\nicefrac{1}{256}$ (which corresponds to $525,313$ nodes) along with a time step size of $\tau = \nicefrac{0.07}{256}$ as our best estimate of an exact solution. The error between the solution to the second order scheme and the exact solution should be less than or equal to $C(\epsilon,T)\cdot \left(\left(\nicefrac{1}{256}\right)^2 + \left(\nicefrac{0.07}{256}\right)^2\right) \approx C(\epsilon,T)\left(1.5\times10^{-5}\right)$ \cite{DWW:2016:SOCH}. We then record the $H^1(\Omega)$ errors with respect to the phase field variable $\phi$ evaluated at the final stopping time of $T=0.7$ in Table \ref{ch-tab-comparison}. We furthermore record the average time to solve per time step. We point out that a mesh size of $h=\nicefrac{1}{32}$ for the second order scheme is comparable to a mesh size of $h=\nicefrac{1}{256}$ for the first order scheme and we clearly see the advantage of the second order scheme.

 \begin{table}[H]
	\centering
	\begin{tabular}{c|cc|cc}
& \multicolumn{2}{ c }{Second Order Method}  &  \multicolumn{2}{| c }{First Order Method}
 \\
$ h $  & $H^1(\Omega)$ Error & Time to Solve (s) & $H^1(\Omega)$ Error & Time to Solve (s)
\\
\hline
$\nicefrac{1}{16}$  & $1.89 \times 10^{-1} $  & $1.09 \times 10^{-1}$ & $6.04 \times 10^{-1}$ & $ 3.54 \times 10^{-2}$
	\\
$\nicefrac{1}{32}$  &  \framebox[1.1\width]{$ 3.22 \times 10^{-2}$ } & \framebox[1.1\width]{$3.53 \times 10^{-1}$ } & $ 3.63 \times 10^{-1}$  & $7.71 \times 10^{-2}$
	\\
$\nicefrac{1}{64}$  &$ 4.59 \times 10^{-3}$  & $1.52 \times 10^0$ & $ 1.85 \times 10^{-1}$  & $2.39 \times 10^{-1}$
	\\
$\nicefrac{1}{128}$  & $6.44 \times 10^{-4} $ & $7.31 \times 10^0$ & $ 9.30 \times 10^{-2}$  & $1.11 \times 10^0 $
	\\
$\nicefrac{1}{256}$  & N/A & N/A & \framebox[1.1\width]{ $ 4.66 \times 10^{-2}$ }  & \framebox[1.1\width]{$5.45 \times 10^0$}
\\
	\hline
	\end{tabular}
	\caption{A comparison of the solver for the Second Order Method to the solver for the First Order Method. The parameters are as follows: $\Omega = (0,1)^2, T=0.7, \varepsilon = 0.03, \tau = \nicefrac{0.07}{256}$.}
	\label{ch-tab-comparison}
	\end{table}
	
	Additionally, a major advantage to both the first and second order schemes considered herein is that they achieve optimal order error estimates in which the mesh and time step sizes may be chosen completely independent of one another. See Remark \ref{rem:stability}. For this experiment, we have chosen a time step size small enough so as not to interfere with the errors presented in Table \ref{ch-tab-comparison}. However, it should be noted that the first order scheme considered in \cite{BDS:2018:CHMGSolver,DFW:2015:CHDS} is, in fact, first order in time and the second order scheme considered in this paper is second order in time. Therefore, we would expect to be able to take larger time step sizes using the second order scheme than when using the first order scheme to achieve comparable error estimates. The effect would be that the first order scheme would take significantly more time steps than the second order scheme in order to achieve a comparable error estimate. We would expect this to have a significant impact in the overall time to solve.
	
	For instance, if we take the time step size equal to a constant multiple of the space step size, such as $\tau=0.07h$, we would not expect the errors above to change much from those listed in Table \ref{ch-tab-comparison}. This fact is demonstrated in Table \ref{ch-tab-comparison-vt}. Additionally, if we again compare similar errors, we see that the first order method would require 2560 times steps but the second order method would only require 320 time steps. The total time to solve is approximately 134 seconds for the second order scheme versus a total time to solve of approximately 241 minutes for the first order scheme.
	
	\begin{table}[H]
	\centering
	\begin{tabular}{c|ccc|ccc}
& \multicolumn{3}{ c }{Second Order Method}  &  \multicolumn{3}{| c }{First Order Method}
 \\
$ h $  & $H^1(\Omega)$ Error & Time to Solve & Time Steps & $H^1(\Omega)$ Error & Time to Solve & Time Steps
\\
\hline
$\nicefrac{1}{16}$  & $ 1.89 \times 10^{-1}$  & $1.23 \times 10^{-1}$ & 160 &  $ 6.04 \times 10^{-1}$& $3.81 \times 10^{-2}$  & 160
	\\
$\nicefrac{1}{32}$  &  \framebox[1.1\width]{$ 3.22 \times 10^{-2}$} & \framebox[1.1\width]{$4.20\times 10^{-1}$} & \framebox[1.1\width]{320} & $ 3.63 \times 10^{-1}$  & $9.80 \times 10^{-2}$ & 320
	\\
$\nicefrac{1}{64}$  & $ 4.62 \times 10^{-3}$  & $1.46 \times 10^0$ & 640 & $ 1.85 \times 10^{-1}$ & $3.21 \times 10^{-1}$ & 640
	\\
$\nicefrac{1}{128}$  & $ 6.68 \times 10^{-4}$ & $7.37 \times 10^0$ & 1280 & $ 9.31 \times 10^{-2}$ & $1.32 \times 10^0$ & 1280
	\\
$\nicefrac{1}{256}$  & N/A & N/A & N/A & \framebox[1.1\width]{$ 4.66 \times 10^{-2}$}  & \framebox[1.1\width]{$5.65 \times 10^0$} &  \framebox[1.1\width]{2560}
\\
	\hline
	\end{tabular}
	\caption{A comparison of the solver for the Second Order Method to the solver for the First Order Method. The parameters are as follows: $\Omega = (0,1)^2, T=0.7, \varepsilon = 0.03, \tau = 0.07h$.}
	\label{ch-tab-comparison-vt}
	\end{table}
	
	In Figure \ref{fig:oval_test}, we show the figures for the initial data mentioned above and the results at the final stopping time $T=0.07$ with $\tau = 0.07h$ and $h=\nicefrac{1}{32}$ for the second order scheme and $h=\nicefrac{1}{256}$ for the first order scheme and observe their similarity.
	
	\begin{figure}[H]
\subfloat{\includegraphics[scale=.4]{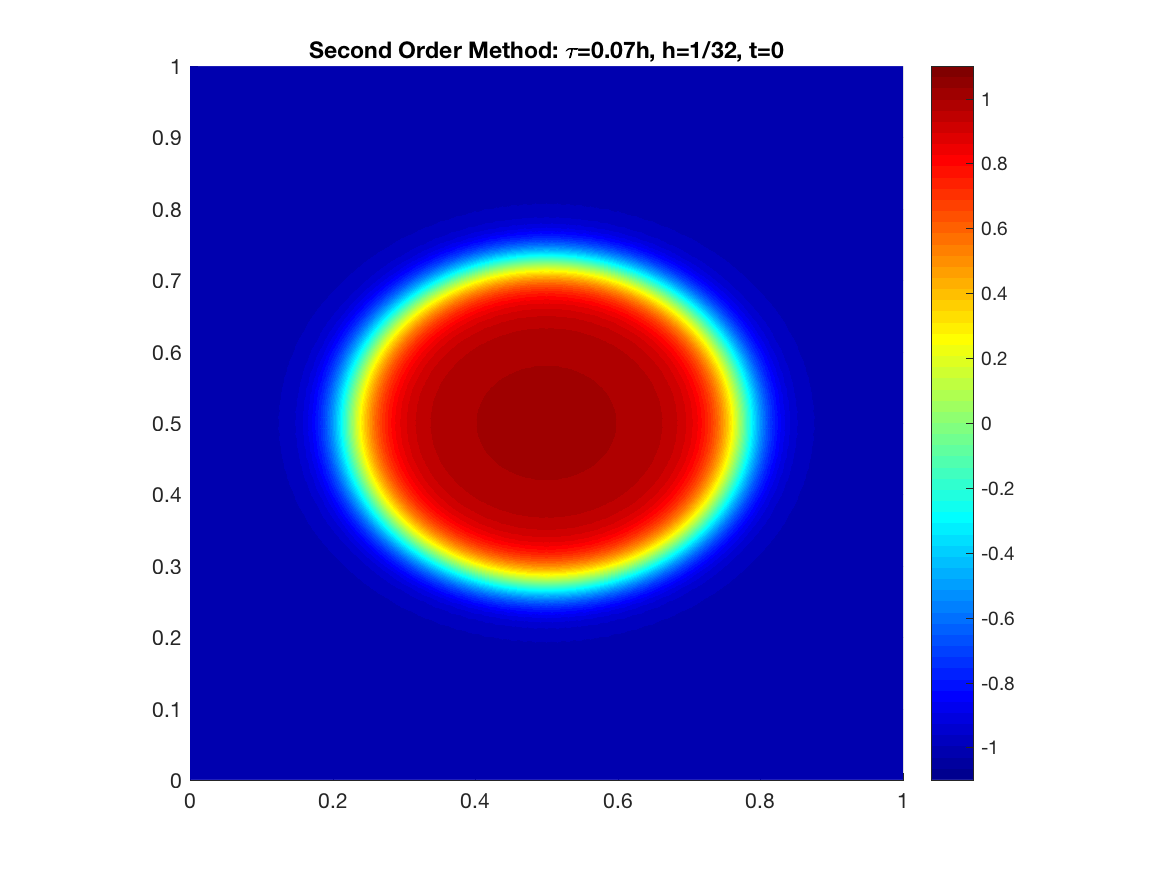}}
\subfloat{\includegraphics[scale=.4]{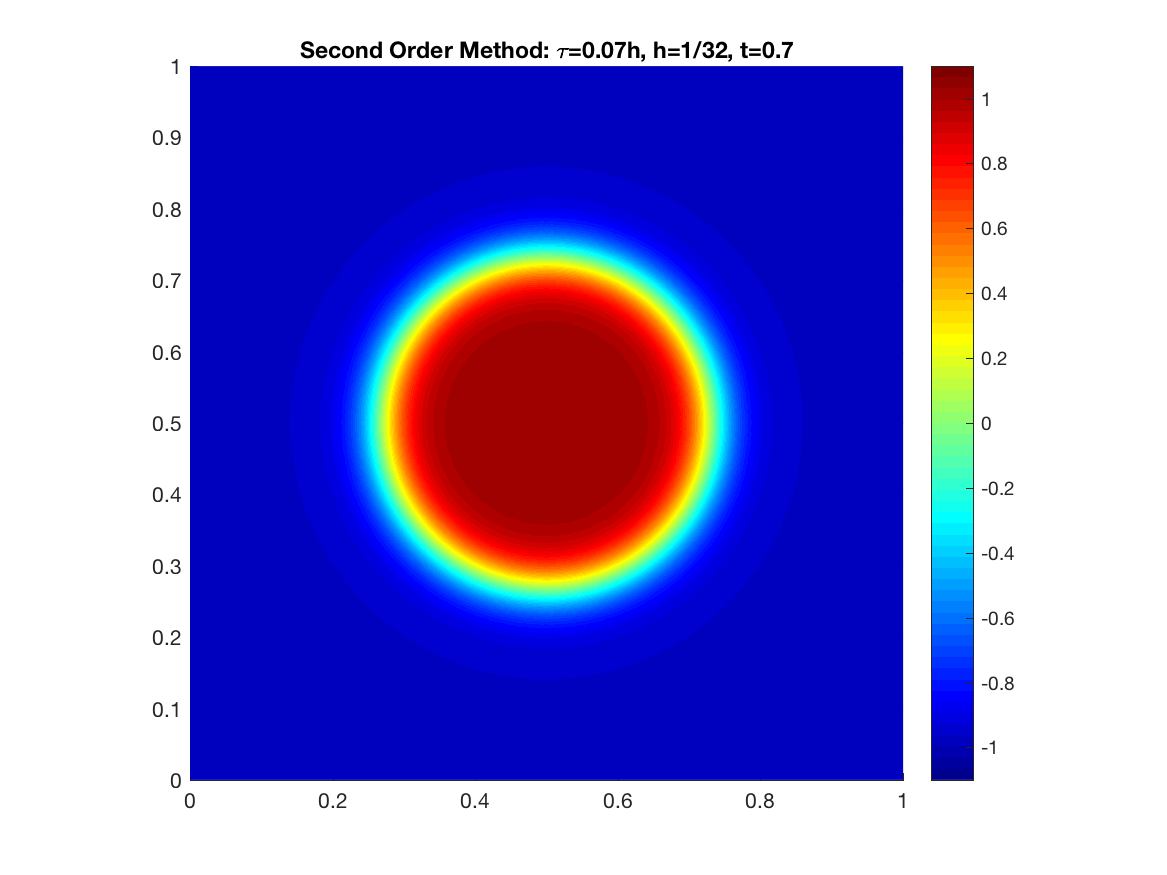}}\\
\subfloat{\includegraphics[scale=.4]{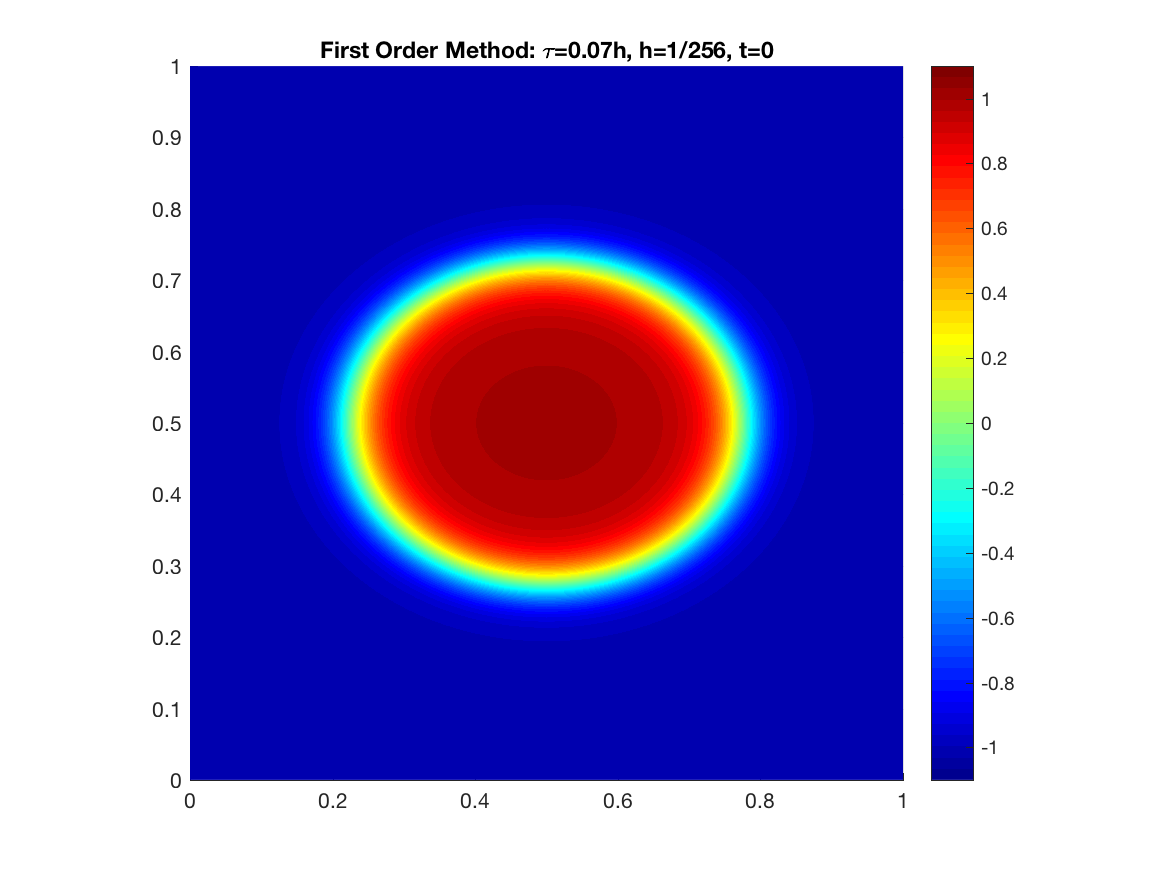}}
\subfloat{\includegraphics[scale=.4]{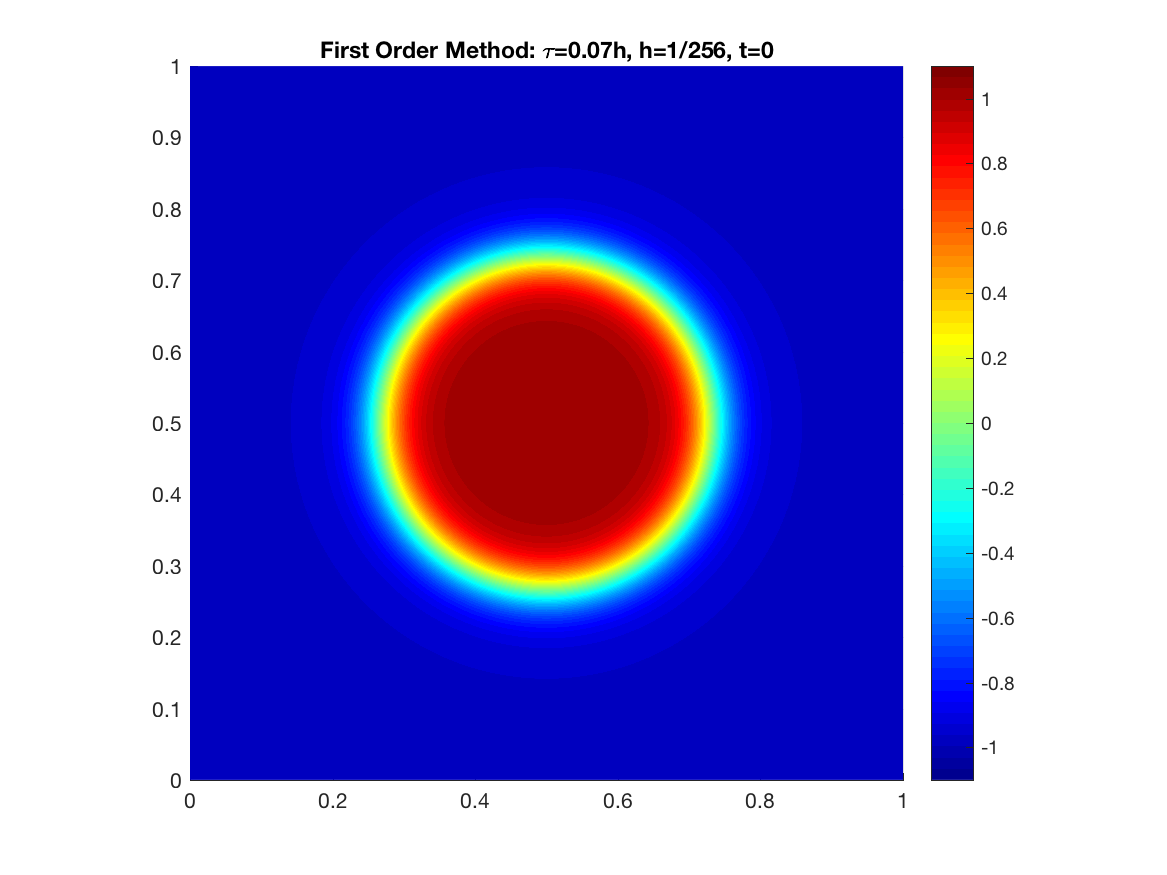}}
\caption{Motion toward a steady state on $(0,1)^2$.
 The times displayed are $t=0, t=0.7$ (from left to right). The top two pictures result from the second order scheme with a mesh size of $h=\nicefrac{1}{32}$ and the bottom two pictures result from the first order scheme with a mesh size of $h=\nicefrac{1}{256}$. The time step size was taken to be $\tau = 0.07h$ in each case.}
\label{fig:oval_test}
\end{figure}

\par
 In the third experiment, we again use the initial data
\begin{equation}
\phi_{h}^0 = \mathcal{I}_h
\Big[\Big(\frac{1}{2}\Big)[1 - \cos(4\pi x_1)][1 - \cos(2\pi x_2))]- 1\Big] ,
\end{equation}
 fix $h=1/64$, a final time
 $T=.04$, $\varepsilon=0.0625$ and $0.001$,  and refine the time step size $\tau$.
 The maximum and average number of the preconditioned MINRES iterations over all
 time steps is displayed in Table~\ref{ch-tab-fixed-h-tau} along with the average solution time per time step.
 The performance is clearly independent of the time step size $\tau$ for an interfacial width parameter of $\varepsilon=0.0625$. When the interfacial width parameter is decrease from $0.0625$ to $0.001$ the solution time roughly triples at worst, indicating again that the performance of the solver only depends mildly on $\varepsilon$.
\begin{table}[H]
	\centering
	\begin{tabular}{c|ccc|ccc}
& \multicolumn{3}{ c }{$\varepsilon = 0.0625$}  &  \multicolumn{3}{| c }{$\varepsilon = 0.001$}
 \\
$\tau$ & Max.~Its. & Avg.~Its. & Avg.~Time to Solve & Max.~Its. & Avg.~Its. & Avg.~Time to Solve
\\
\hline
$\nicefrac{.02}{8}$ & 54 & 50 & 1.62 & 126 & 53 & 1.87
	\\
$\nicefrac{.02}{16}$ & 54 & 50 & 1.58 & 132 & 55 & 1.67
	\\
$\nicefrac{.02}{32}$ & 54 & 50 & 1.48 & 139 & 58 & 1.64
	\\
$\nicefrac{.02}{64}$ & 55 & 50 & 1.38 & 141 & 72 &  2.09
	\\
$\nicefrac{.02}{128}$ & 55 & 49 & 1.36 & 158 & 86 & 2.21
	\\
$\nicefrac{.02}{256}$ & 54 & 47 & 1.24 & 173 & 98 & 2.41
	\\
$\nicefrac{.02}{512}$ & 54 & 46 & 1.20 & 171 & 105 & 2.71
	\\
$\nicefrac{.02}{1024}$ & 52 & 43 & 1.27 & 164 & 108 & 2.76
\\
$\nicefrac{.02}{2048}$ & 47 & 42 & 1.26 & 164 & 115 & 2.87
\\
$\nicefrac{.02}{4056}$ & 42 & 37 & 1.20 & 169 & 124 & 3.08
\\
	\hline
	\end{tabular}
	\caption{The maximum and average number of preconditioned MINRES iterations over
all time steps along with the average solution time per time step as the time step is refined
 ($\Omega = (0,1)^2, h = \nicefrac{1}{64}$, $T = 0.04$, $\varepsilon = 0.0625$
  (left), $\varepsilon = 0.001$ (right)).}
	\label{ch-tab-fixed-h-tau}
	\end{table}

\par
 In the fourth experiment, we show that our method accurately demonstrates motion towards a steady state.
 We take the initial conditions such that $\phi = -1$ outside of the cross and $\phi=1$ inside of the cross. The cross is constructed using the lines $x_1=0.3, 0.4, 0.6, 0.7, x_2 = 0.3, 0.4, 0.6, 0.7$. Additionally, we take $h=1/64$, $\tau=0.002/64$ and $\varepsilon=0.01$.  The surface plots for $\phi$
 at $t=0, t=10\tau, t=180\tau$, $t=500\tau, t=980\tau$ and $t=2100\tau$ are displayed in
 Figure~\ref{fig:cross-matlab}.

\begin{figure}[H]
\subfloat{\includegraphics[width = 2in]{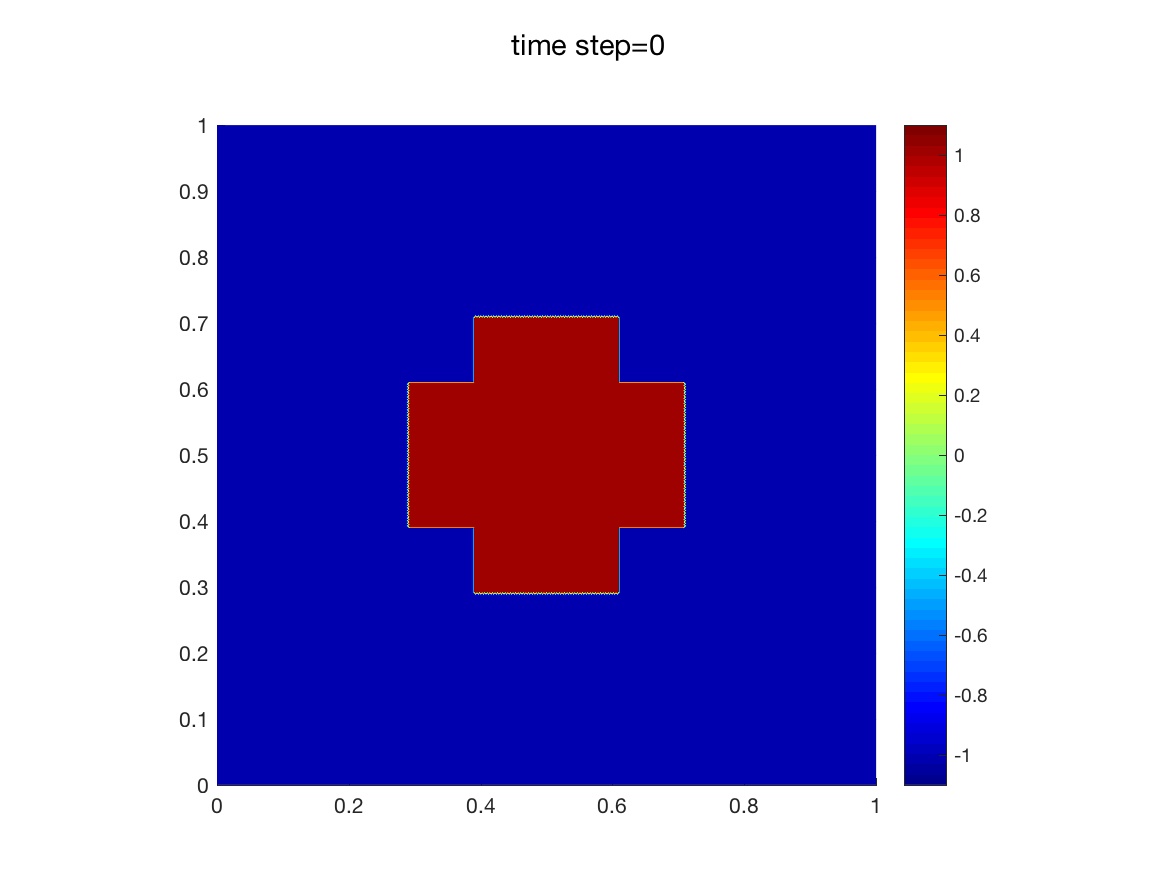}}
\subfloat{\includegraphics[width = 2in]{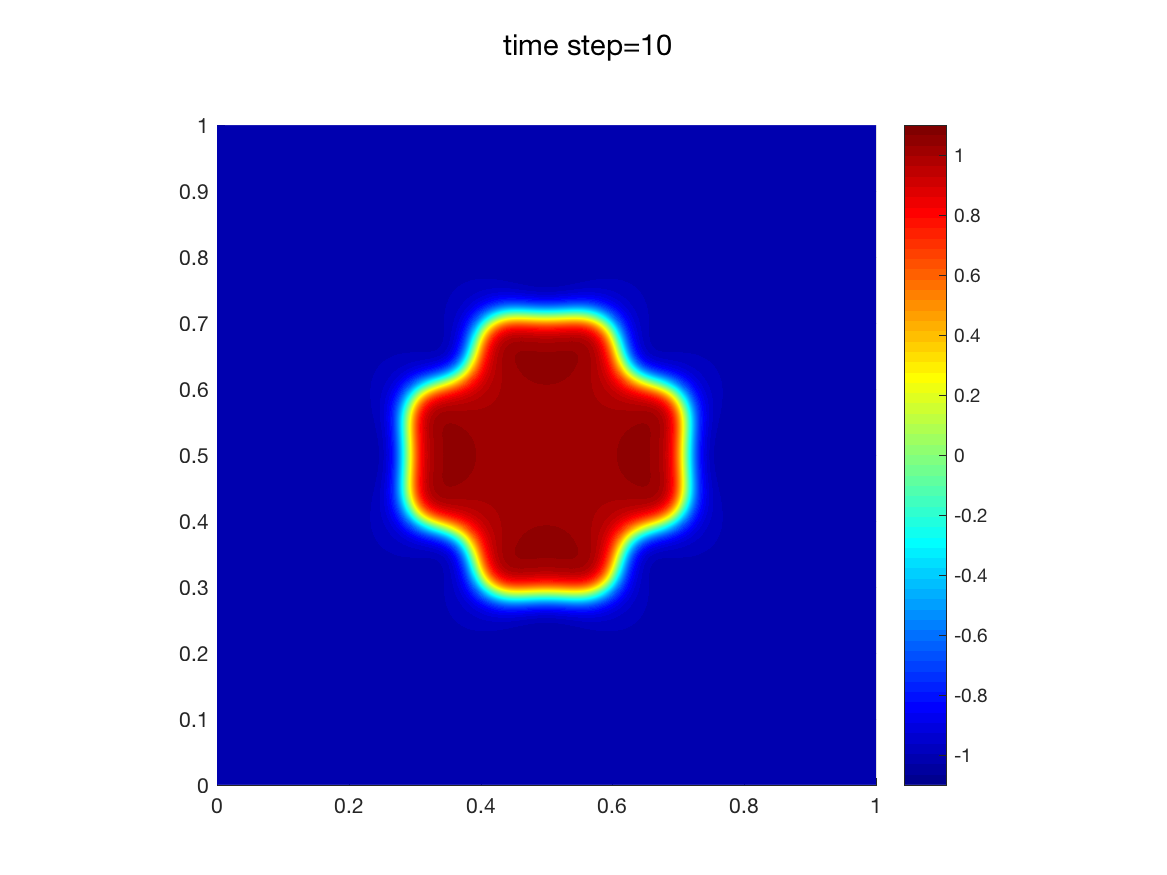}}
\subfloat{\includegraphics[width = 2in]{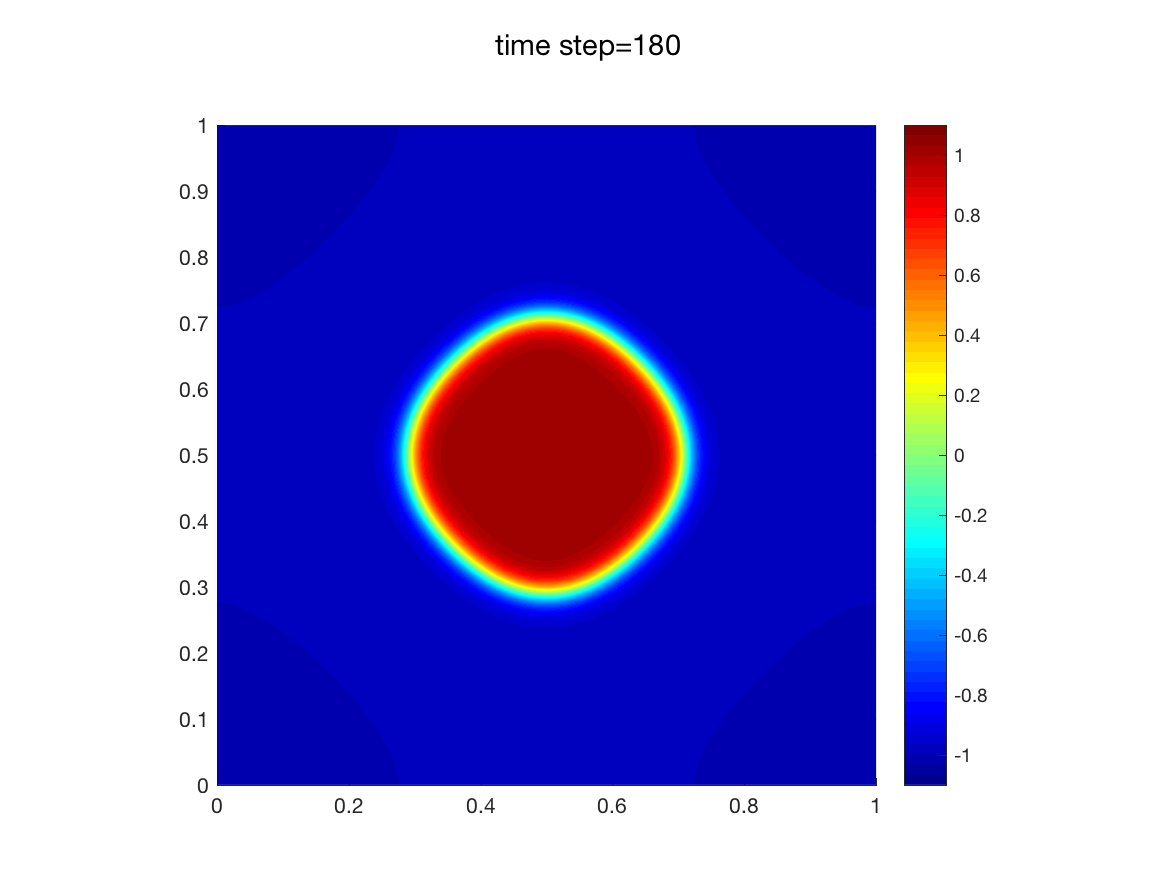}} \\
\subfloat{\includegraphics[width = 2in]{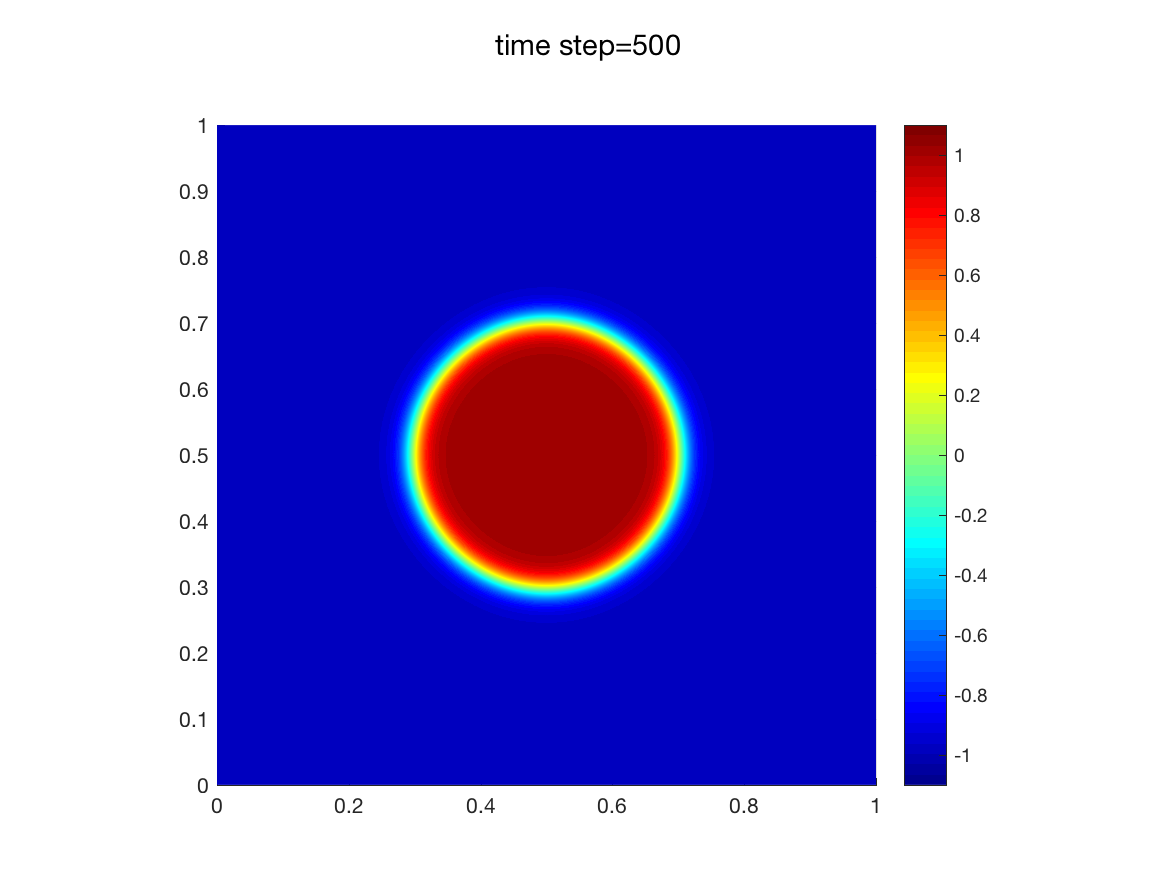}}
\subfloat{\includegraphics[width = 2in]{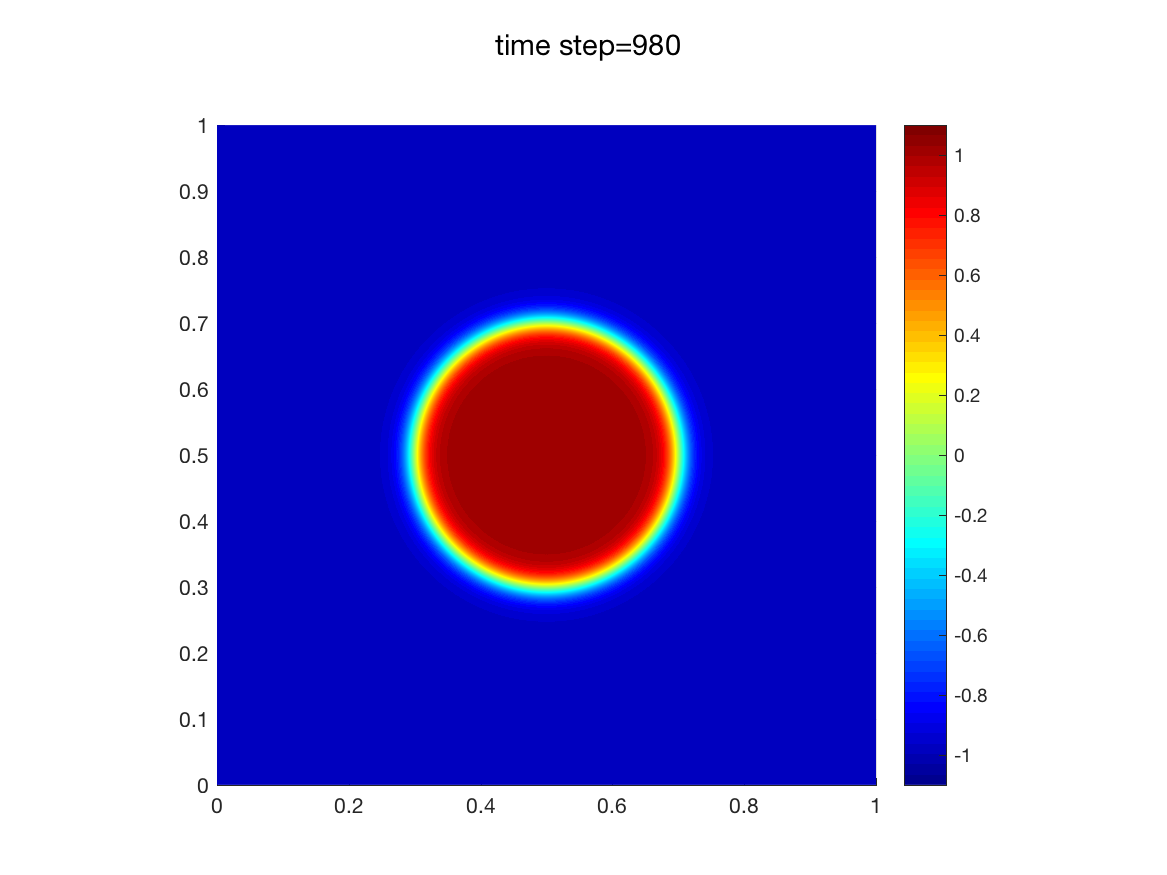}}
\subfloat{\includegraphics[width = 2in]{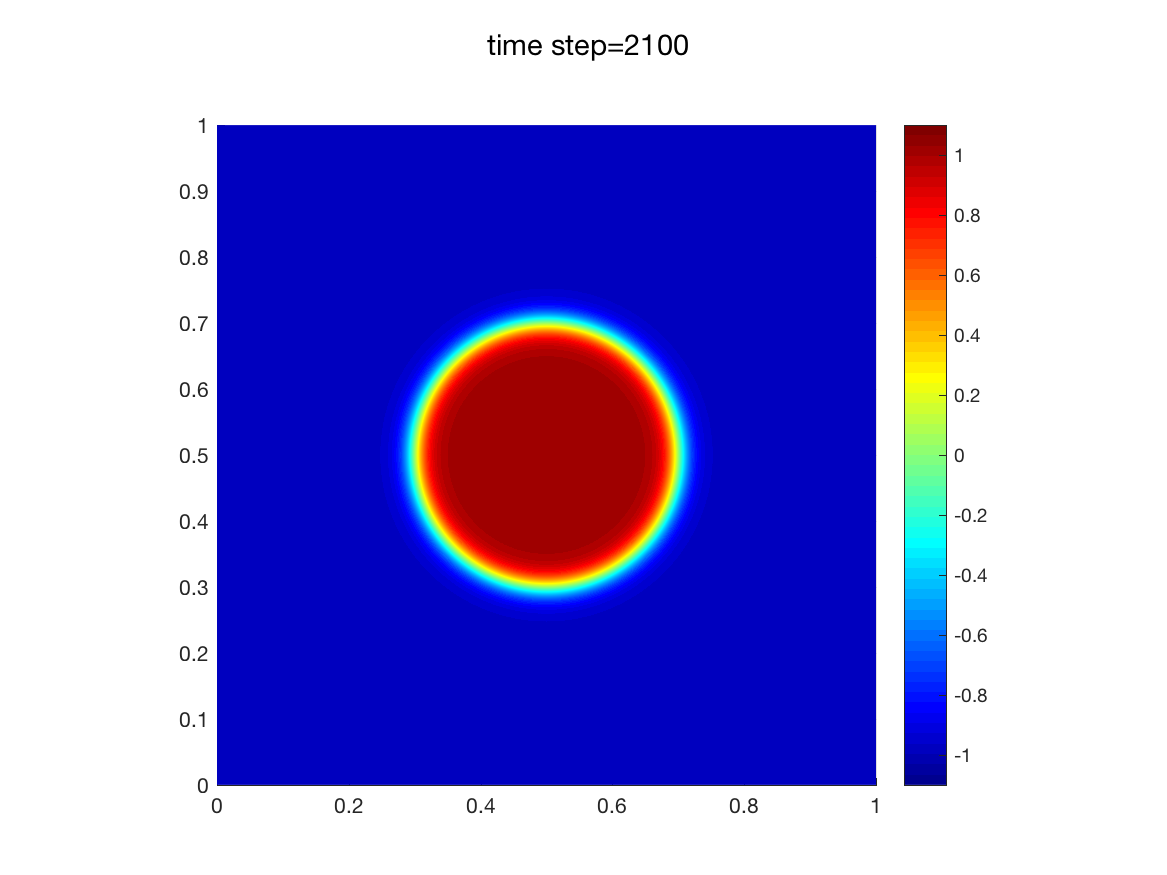}}\\
\caption{Motion towards a steady state of a binary fluid on $(0,1)^2$.
 The times displayed are $t=0, t=10\tau, t=180\tau$ (top from left to right) and $t=500\tau, t=980\tau, t=2100\tau$ (bottom from left to right).}
\label{fig:cross-matlab}
\end{figure}

\par
 In the final experiment, we solve the Cahn-Hilliard equation on the unit cube $\Omega = (0,1)^3$
  with an
 initial condition $$\phi_{h}^0 = \mathcal{I}_h
\left[-1.01\tanh\left(\frac{\nicefrac{(x_1-0.5)^2)}{0.075}+ \nicefrac{(x_2-0.5)^2}{0.05} + \nicefrac{(x_3-0.5)^2}{0.05}-1}{2\sqrt{\varepsilon}}\right)\right]$$ 
 that represents a droplet elongated along the $x_1$-axis,
 as depicted  in Figure \ref{fig:3D-minres}.  
 The initial mesh $\mathcal{T}_0$
 consists of six tetrahedrons and the meshes $\mathcal{T}_1, \mathcal{T}_2, \cdots$
 are obtained from $\mathcal{T}_0$ by uniform refinements.
 We take $\varepsilon=0.03$, $\tau=0.002/32$, and a final time $T=0.1$ and
 refine the mesh four times so that $h=\nicefrac{\sqrt{3}}{32}$.
\par

	Isocap plots for $\phi$ at $t=0, t=0.05,$ and $t=0.1$ are displayed in Figure \ref{fig:3D-minres}. We note that the average time to solve per time step was approximately $18.82$s and the average number of MINRES iterations was approximately 42.
	
\begin{figure}[H]
\subfloat{\includegraphics[width = 2in]{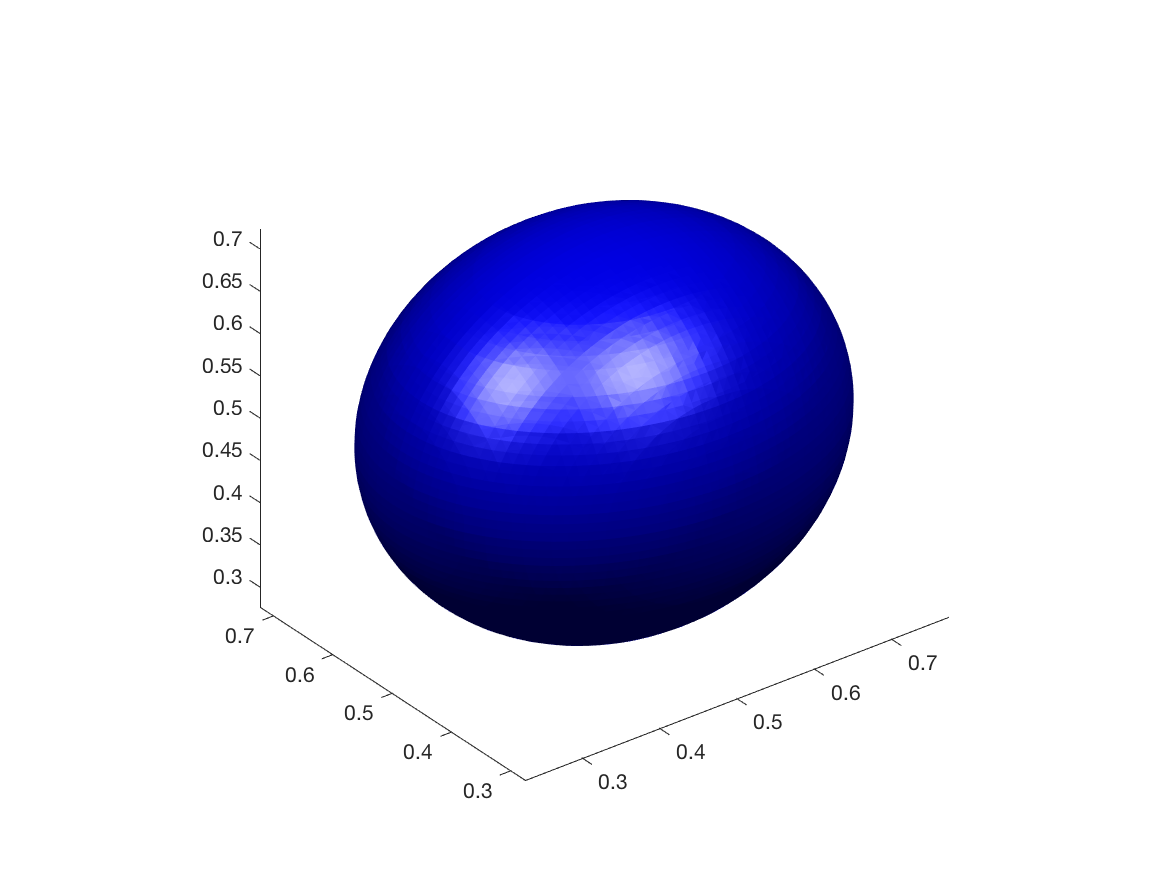}}
\subfloat{\includegraphics[width = 2in]{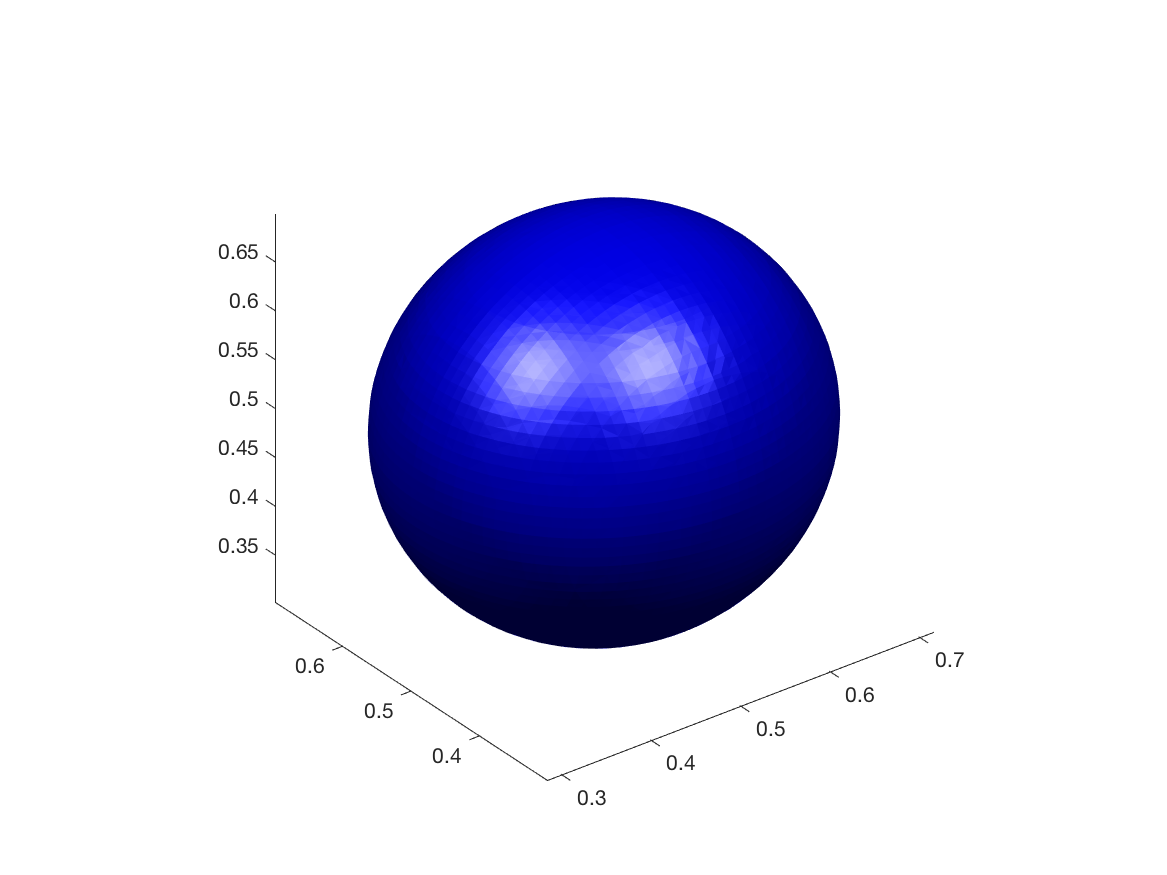}}
\subfloat{\includegraphics[width = 2in]{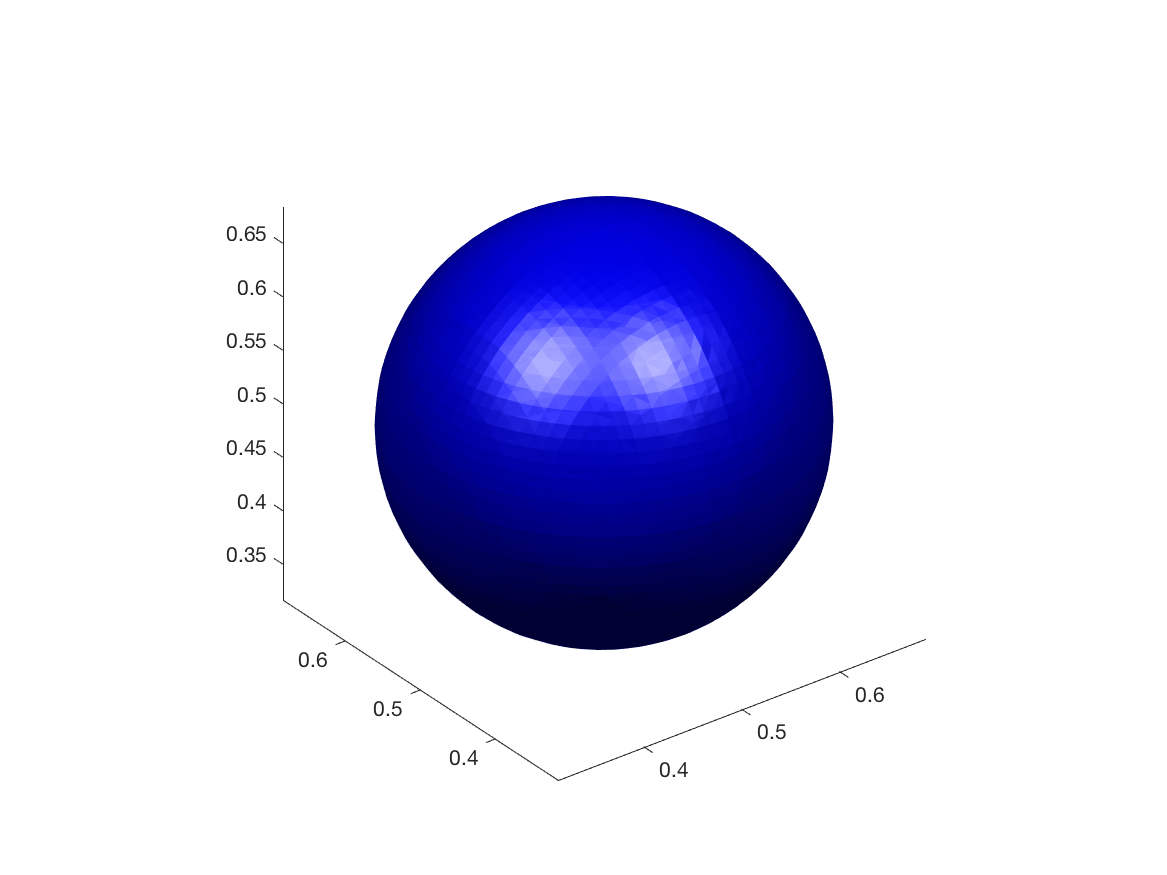}}
\caption{Motion towards a steady state on $(0,1)^3$ with $h=\nicefrac{\sqrt{3}}{32}$.
 The times displayed are $t=0, t=0.05, t=0.1$ .}
\label{fig:3D-minres}
\end{figure}
\par
We similarly compared our 3D results with $h=\nicefrac{\sqrt{3}}{32}$ to that of the direct solver using MATLAB's backslash command whereby considerable computational savings is clearly observed. Specifically, in the test using our solver, 1600 time steps were completed in approximately 8.5 hours whereas, in the test using MATLAB's backslash command, 15 minutes was required to complete only a single time step and the completion of the numerical experiment took a little more than 16 days.

\section{Conclusion}\label{sec:Conclusion}

This paper has been devoted to the development of a robust solver for a second order (in time and space) mixed finite element method for the Cahn-Hilliard equation where in each time step the Jacobian system for the Newton iteration is solved by a preconditioned MINRES algorithm with a block diagonal multigrid preconditioner. The advantages of the solver are demonstrated by several numerical experiments.
\par
We are hopeful that the methodology developed in this paper can be adapted for coupled systems that involve the Cahn-Hilliard equation, such as the Cahn-Hilliard-Navier-Stokes system. In particular, a similar mixed finite element for the Cahn-Hilliard-Navier-Stokes system was developed in \cite{DWWW:2017:SOCHNS} and the investigation of a solver for this particular scheme is an obvious next step and is the topic of an ongoing research project.

\section*{Acknowledgement}
Portions of this research were conducted with high performance
computational resources provided by Louisiana State University
(http://www.hpc.lsu.edu).
 We would also like to thank Shawn Walker for his valuable advice regarding the
 FELICITY/C++ Toolbox for MATLAB.

	\end{document}